\documentclass[10pt,reqno]{amsart}
\usepackage[T1]{fontenc}
\usepackage{lmodern}
\usepackage{amsmath,amssymb,amsthm,mathtools}
\usepackage[a4paper,margin=30mm]{geometry}
\usepackage{microtype}
\usepackage{enumitem}
\usepackage{xcolor}
\usepackage{hyperref}
\hypersetup{colorlinks=true,linkcolor=blue!50!black,citecolor=blue!50!black,urlcolor=blue!50!black,pdftitle={Harmonic Torsion of Algebraic Generalized Ricci Solitons and the Structure of Solvable Extensions},pdfauthor={Huibin Chen, Zhiqi Chen, Fuhai Zhu}}
\newcommand{\R}{\mathbb R}
\newcommand{\g}{\mathfrak g}
\newcommand{\n}{\mathfrak n}
\newcommand{\af}{\mathfrak a}
\newcommand{\h}{\mathfrak h}
\newcommand{\Ecal}{\mathcal E}
\newcommand{\Fcal}{\mathcal F}

\DeclareMathOperator{\Ric}{Ric}
\DeclareMathOperator{\scal}{scal}

\DeclareMathOperator{\tr}{tr}
\DeclareMathOperator{\ad}{ad}
\DeclareMathOperator{\Der}{Der}
\DeclareMathOperator{\Aut}{Aut}
\DeclareMathOperator{\Sym}{Sym}
\DeclareMathOperator{\Skew}{Skew}

\DeclareMathOperator{\pr}{pr}
\newtheorem{maintheorem}{Theorem}

\newtheorem{theorem}{Theorem}[section]
\newtheorem{proposition}[theorem]{Proposition}
\newtheorem{lemma}[theorem]{Lemma}
\newtheorem{corollary}[theorem]{Corollary}
\theoremstyle{definition}

\newtheorem{example}[theorem]{Example}
\theoremstyle{remark}
\newtheorem{remark}[theorem]{Remark}
\numberwithin{equation}{section}
\allowdisplaybreaks[2]
\title[Harmonic torsion and solvable extensions]{Harmonic Torsion of Algebraic Generalized Ricci Solitons and the Structure of Solvable Extensions}
\author{Huibin Chen}
\address{Institute of Mathematics, School of Mathematical Sciences, Nanjing Normal University, Nanjing 210023, P.R. China}
\email{chenhuibin@njnu.edu.cn}
\author{Zhiqi Chen}
\address{School of Mathematics and Statistics, Guangdong University of Technology, Guangzhou 510520, P.R. China}
\email{chenzhiqi@gdut.edu.cn}
\author{Fuhai Zhu}
\address{Department of Mathematics, Nanjing University, Nanjing 210023, P.R. China}
\email{zhufuhai@nju.edu.cn}
\date{}
\begin{document}
\begin{abstract}
We prove harmonicity for nonshrinking algebraic generalized Ricci solitons and for all algebraic generalized Ricci solitons on solvable Lie algebras. In the solvable setting, we characterize expanding solitons by normal extensions of generalized nilsolitons, and derive a sharp rank bound.
\end{abstract}
\maketitle
\tableofcontents

\section{Introduction}\label{sec:introduction}
Generalized geometry studies structures on $TM\oplus T^*M$ together with its natural pairing of split signature. Courant's Dirac structures \cite{Co1990}, the Lie bialgebroids of Liu, Weinstein and Xu \cite{LiWeXu1997}, and the generalized complex structures introduced by Hitchin and Gualtieri \cite{Hi2003,Gu2011} provide the basic framework. For Riemannian questions, an exact Courant algebroid is described in an isotropic splitting by a closed three-form $H_0$. Its bracket is
\begin{equation}\label{eq:intro-dorfman}
[X+\xi,Y+\zeta]_{H_0}
=[X,Y]+\mathcal L_X\zeta-\iota_Yd\xi+H_0(X,Y,\cdot).
\end{equation}
The isomorphism class is determined by $[H_0]\in H^3_{\mathrm{dR}}(M,\R)$, and a change of splitting is given by a two-form \cite{SeWe2001,Se2017}. A generalized metric is represented by a Riemannian metric $g$ and a two-form $b$. The three-form entering its curvature is $H=H_0+db$. Reduction and automorphisms of these structures are studied in \cite{BuCaGu2007,RuTi2020}.

The generalized Ricci flow evolves a metric and a closed three-form by
\begin{equation}\label{eq:intro-grf}
\left\{
\begin{aligned}
\frac{\partial g}{\partial t}&=-2\Ric_g+\tfrac12H_g^2,\\
\frac{\partial H}{\partial t}&=-d\delta_gH,\\
dH&=0,
\end{aligned}
\right.
\end{equation}
where $H_g^2(X,Y)=\sum_{i,j}H(X,e_i,e_j)H(Y,e_i,e_j)$ and $\delta_g$ is the geometric codifferential. It reduces to Hamilton's Ricci flow when $H=0$ \cite{Ha1982}. The coupled system arises in the one-loop renormalization of nonlinear sigma models with a $B$-field. Oliynyk, Suneeta and Woolgar gave a variational formulation in \cite{OlSuWo2006}. Streets developed regularity and entropy methods in \cite{St2008,St2023}, and Li obtained curvature estimates in \cite{Li2012}. Pluriclosed flow, introduced by Streets and Tian \cite{StTi2010}, gives a Hermitian source of this system through the gauge equivalence in \cite{StTi2013}. Its relation with generalized K\"ahler geometry is developed in \cite{StTi2012,Gu2014,GaJoSt2023}. Generalized Ricci curvature on Courant algebroids and its compatibility with duality are studied in \cite{St2017,SeVa2017,Ga2019,SeVa2020}. The recent work \cite{StStVa2026} establishes short-time existence for generalized metrics coupled to half-densities over compact manifolds. We refer to \cite{GaSt2021} for a systematic account of the exact Courant-algebroid theory.

Stationary solutions lead to Bismut--Ricci-flat geometry. For the metric connections $\nabla^\pm=\nabla^g\pm H^\sharp/2$, with $g(H^\sharp(X,Y),Z)=H(X,Y,Z)$, the full Ricci tensors are
\begin{equation}\label{eq:intro-bismut-ricci}
\Ric^{\nabla^\pm}=\Ric_g-\tfrac14H_g^2\mp\tfrac12\delta_gH.
\end{equation}
Bismut's Hermitian connection \cite{Bi1989} is an important instance of this construction. A closed pair is called Bismut--Ricci-flat, or BRF, when $\Ric_g=H_g^2/4$ and $\delta_gH=0$. Compact Lie groups with a bi-invariant metric and its Cartan three-form give basic examples. The distinction between Bismut flatness and Ricci flatness is illustrated by \cite{WaYaZh2020,PoRa2023}. Further compact homogeneous BRF pairs are constructed in \cite{PoRa2024,LaWi2023a}, with the correction in \cite{LaWi2025}. The associated invariant harmonic forms are studied in \cite{LaWi2023b}. Stability and deformation questions are addressed in \cite{RaVe2021,Le2024,Le2025}. On Lie groups, generalized Einstein classifications also depend on the divergence chosen in the Courant description \cite{CoKr2023,CoFrGa2024}. In particular, \cite[Corollary~3.3]{CoFrGa2024} gives torsion vanishing and flatness in the Riemannian solvable case under its generalized Einstein hypotheses.

For a homogeneous flow, varying the bracket while keeping the inner product fixed converts the evolution into a finite-dimensional orbit problem. Lauret's work on nilsolitons and moment maps \cite{La2001,La2003,La2013}, Heber's Einstein structure theory \cite{He1998}, and the structure theory of solvsolitons \cite{La2010,La2011} explain the role of the nilradical in this method. Lafuente and Lauret established corresponding structural results for homogeneous Ricci solitons \cite{LaLa2014}. In the generalized setting, Arroyo and Lafuente studied homogeneous pluriclosed flow \cite{ArLa2019}, and Paradiso formulated a bracket flow for nilpotent generalized Ricci flow \cite{Pa2021}. Fusi, Lafuente and Stanfield use the group $\mathrm{GL}(\g)\ltimes\Lambda^2\g^*$ to describe arbitrary invariant generalized metrics \cite{FuLaSt2024}. Their results include global existence on solvmanifolds and convergence of rescaled nilpotent solutions with harmonic initial torsion to expanding generalized solitons. Their second paper \cite{FuLaSt2026} uses the associated quadratic Lie algebra to obtain further convergence results under a nonnegative Killing-form assumption.

Generalized Ricci solitons are the self-similar solutions of \eqref{eq:intro-grf}. We use simultaneous scaling of the metric and three-form, as in \cite{FuLaSt2024,PoRa2025}.  The rotationally symmetric steady examples in \cite{PoRa2025} illustrate the use of ordinary differential equations outside the homogeneous setting. Here the self-similarity is generated by automorphisms. On a finite-dimensional real metric Lie algebra with positive-definite inner product, our algebraic system is
\begin{equation}\label{eq:soliton}
\left\{
\begin{aligned}
&R:=\Ric-\tfrac14H^2=\lambda I+D,\\
&d\delta H=2\lambda H-\rho(D)H,\qquad D=D^*\in\Der(\g),\\
&dH=0.
\end{aligned}
\right.
\end{equation}
The action $\rho$ is negative insertion, as specified in \eqref{eq:form-action}. The solution is expanding, steady or shrinking according as $\lambda<0$, $\lambda=0$ or $\lambda>0$. This is the algebraic formulation of \cite[Definition~7.3 and Proposition~7.4]{FuLaSt2024}. 

The torsion is harmonic if $dH=0$ and $\delta H=0$. On invariant forms, $\delta=d^\dagger+\iota_U$, where $d^\dagger$ is the exterior adjoint of the Chevalley--Eilenberg differential and $\langle U,X\rangle=\tr\ad_X$. Thus coclosedness requires care even at the algebraic level. Chen, Chen and Zhu proved harmonicity for all generalized nilsolitons in \cite{ChChZh2026a}, answering \cite[Question~7.12]{FuLaSt2024}. Their subsequent defect identity \cite{ChChZh2026b} gives, with $\eta=\delta H$ and $\eta(X,Y)=\langle K_\eta X,Y\rangle$,
\begin{equation}\label{eq:intro-killing-obstruction}
\tr(K_\eta^*K_\eta B)\ge0\quad\Longrightarrow\quad\eta=0,
\end{equation}
where $B$ is the Killing operator. Fusi, Lafuente and Stanfield prove harmonicity in the unimodular nonnegative-Killing case in \cite[Corollary~B]{FuLaSt2026}.   Our first result removes this Killing sign restriction for nonshrinking solitons.

\begin{maintheorem}\label{thm:harmonicity}
Let $\g$ be a finite-dimensional real metric Lie algebra, and suppose that $H$, $\lambda$ and $D$ satisfy \eqref{eq:soliton}. If $\lambda\le0$, then $\delta H=0$. If $\lambda=0$, then also $\Ric=H^2/4$ and $D=0$. In particular, every expanding or steady algebraic generalized Ricci soliton has harmonic torsion. Every non-unimodular algebraic generalized Ricci soliton is expanding and hence has harmonic torsion. If $\g$ is solvable, harmonicity holds for every value of $\lambda$.
\end{maintheorem}

For expanding solitons, the proof also shows that $\af=\n^\perp$ is a reductive subalgebra and $H(\af,\af,\g)=0$, where $\n$ is the nilradical. The steady assertion follows already from the scalar evolution in \cite[Section~6, equation~(44)]{FuLaSt2024}, following \cite{St2023}. The proof combines elementary algebraic identities with real geometric invariant theory on a nilpotent quadratic quotient.

On solvable Lie algebras, the reductive complement is abelian, and the remaining horizontal torsion also vanishes. This gives both structure and reconstruction.

\begin{maintheorem}\label{thm:structure}
Let $\g$ be a solvable metric Lie algebra satisfying \eqref{eq:soliton}, and set $\af=\n^\perp$. Then
\begin{equation}\label{eq:main-structure}
[\af,\af]=0,\qquad H(\af,\g,\g)=0,\qquad
[\ad_A,(\ad_A)^*]=0\quad(A\in\af).
\end{equation}
If the soliton is expanding, its restriction to $\n$ is a generalized nilsoliton with the same constant. Conversely, all expanding solvable solutions arise by extending generalized nilsoliton data through an abelian subalgebra $\af\subset\Der(\n)$ of normal operators preserving $H$, for which $A\mapsto S(A):=\Sym A$ is injective. The horizontal metric is given by $\langle A,B\rangle=-\lambda^{-1}\tr(S(A)S(B))$.
\end{maintheorem}

The precise reconstruction, including the mean-curvature correction to the soliton derivation, is stated in Theorem~\ref{thm:extension}. It extends the classical normal construction of \cite[Theorem~4.8]{La2011}. Normality permits skew-symmetric components, so complete solvability is not a hypothesis. Mucciante proves the expanding sign for solvable solitons with nonzero torsion in \cite[Theorem~4.2]{Mu2026}, and treats specified unimodular codimension-one extensions with abelian nilradical in Propositions~4.4--4.5. Theorem~\ref{thm:structure} allows arbitrary solvable nilradicals. Corollary~\ref{cor:expanding-solvable} also includes the zero-torsion boundary. A solvable solution with nonzero torsion or nonflat metric is expanding, every steady solution is flat with zero torsion, and a shrinking solution has abelian underlying algebra and zero torsion.

We describe the main tools and their relation to earlier work. The formulation of the soliton equations on the twisted cotangent Lie algebra $\g\oplus\g^*$ comes from \cite[Proposition~7.4 and Lemma~7.6]{FuLaSt2024}. We use the expression for the generalized Ricci operator in \cite[Theorem~3.1, equation~(18)]{FuLaSt2026}, with its normalization verified in Proposition~\ref{prop:cotangent-curvature}. The trace identities used below extend arguments for classical homogeneous Ricci solitons in \cite[Lemmas~2.3, 2.6 and~4.2, Remark~2.7 and Proposition~4.1]{LaLa2014}. Our main step is to construct a positive-definite operator on the nilpotent quadratic quotient
\begin{equation}\label{eq:introduction-quotient}
W=(\n\oplus\g^*)/\operatorname{ann}(\n).
\end{equation}
We retain the block determined by $\eta=\delta H$ until the equality argument is complete. The quotient is necessary because the corresponding subspace $\n\oplus\n^*$ need not be a Lie subalgebra.

The real geometric invariant theory of B\"ohm and Lafuente \cite{BoLa2020} provides a stratum label $\beta$ for the nonzero quotient bracket, with the given positive-definite inner product. Setting $s=\|\beta\|^2>0$, we prove that $C=\beta+sI$ is positive definite and extend it to a positive-semidefinite operator $\Ecal$ on $\g\oplus\g^*$. Let $\nu$ denote the twisted cotangent bracket and $M_\nu$ its moment operator, defined by \eqref{eq:moment-operator}. For $\lambda=-c<0$, the derivation $L$ in \eqref{eq:double-derivation} and the mean-curvature vector $U$ yield
\begin{equation}\label{eq:introduction-square}
0=\left\|\Fcal-\frac cs\Ecal\right\|^2
+\frac{2c}{3s}\tr(M_\nu\Ecal),
\qquad \Fcal=\Sym(L+\ad_U^\nu).
\end{equation}
We decompose the Cartan three-form $\tau(X,Y,Z)=q(\nu(X,Y),Z)$ into five orthogonal components, where $q$ is the natural invariant pairing on $\g\oplus\g^*$. This proves that the trace term is nonnegative. The component determined by the bracket on $\g/\n$ contributes zero to this trace pairing. Equality gives $\eta(\n^\perp,\g)=0$, and hence $K_\eta^*K_\eta B=0$. Equation~\eqref{eq:intro-killing-obstruction} then proves harmonicity. A separate calculation treats the zero quotient bracket, including $W=0$. For solvable $\g$, a further commutator identity proves normality and $H(\n^\perp,\g,\g)=0$.

The remaining results follow from the criterion for extending generalized nilsolitons by commuting normal derivations. When $H\ne0$, we obtain the sharp bound $\dim\af\le\dim\n-1$ and classify the equality cases. For fixed generalized nilsoliton data, extensions by $k$-dimensional spaces of commuting symmetric derivations are parametrized, up to orthogonal Lie algebra isomorphisms preserving $H$, by a quotient of a Grassmannian by a finite group. This extends Will's description for classical solvsolitons \cite[Section~3, equation~(5)]{Wi2011} to the simultaneous stabilizer of the bracket and the three-form. The identity $H(\n^\perp,\g,\g)=0$ also restricts which invariant cohomology classes can occur. Explicit examples include Lie algebras that are not completely solvable and whose Killing form takes negative values.

The paper is organized as follows. Section~\ref{sec:preliminaries} fixes notation, establishes preliminary identities, recalls the harmonicity criterion from \cite{ChChZh2026b}, and treats compact semisimple Lie algebras with bi-invariant metrics. Section~\ref{sec:cotangent} verifies the formula for the generalized Ricci operator on the twisted cotangent Lie algebra and constructs the nilpotent quadratic quotient and its positive-definite operator. Section~\ref{sec:global} proves Theorem~\ref{thm:harmonicity}, including the case of a zero quotient bracket. Section~\ref{sec:structure} proves Theorem~\ref{thm:structure} and the extension criterion. Section~\ref{sec:rank-moduli} gives the sharp rank bound and parametrizes symmetric extensions.

\section{Algebraic solitons and elementary reductions}\label{sec:preliminaries}
In this section, we fix the conventions, establish the preliminary identities and recall the harmonicity criterion used below. We also prove harmonicity for compact semisimple Lie algebras with bi-invariant metrics.

All vector spaces and Lie algebras are finite-dimensional and real. Adjoints are taken with respect to the positive-definite metric. For an endomorphism $A$, set $\Sym A=(A+A^*)/2$ and $\Skew A=(A-A^*)/2$. Endomorphisms carry the Hilbert--Schmidt inner product $\langle A,B\rangle=\tr(AB^*)$ and norm $\|A\|$. We distinguish the usual exterior norm $\|\omega\|_\wedge$ from the full contraction norm $|\omega|$ of a form. In degree $p$,
\begin{equation}\label{eq:norm-convention}
|\omega|^2=p!\|\omega\|_\wedge^2,\qquad
\langle d\alpha,\omega\rangle=(p+1)\langle\alpha,d^\dagger\omega\rangle
\quad(\alpha\in\Lambda^p\g^*).
\end{equation}
The pairings in the second identity are full contractions in their respective degrees. All unmarked form pairings below use this convention. The operator $d^\dagger$ is the adjoint of the Chevalley--Eilenberg differential for the exterior inner product.

The mean-curvature vector $U$ is defined by $\langle U,X\rangle=\tr\ad_X$. On left-invariant forms the geometric codifferential is
\begin{equation}\label{eq:codifferential}
\delta=d^\dagger+\iota_U.
\end{equation}
To verify \eqref{eq:codifferential}, choose a left-invariant orthonormal frame $e_1,\ldots,e_{\dim\g}$ with dual coframe $e^1,\ldots,e^{\dim\g}$, and let $\nabla$ be the Levi-Civita connection. On left-invariant forms, each $\nabla_{e_i}$ is a skew-adjoint endomorphism for the exterior inner product, by metric compatibility. Since the connection is torsion-free, $d=\sum_i e^i\wedge\nabla_{e_i}$. Taking its exterior adjoint and using the Leibniz rule gives, for any left-invariant form $\omega$,
\begin{equation}\label{eq:codifferential-adjoint}
d^\dagger\omega=-\sum_i\nabla_{e_i}(\iota_{e_i}\omega) =-\sum_i\iota_{e_i}\nabla_{e_i}\omega
-\iota_{\sum_i\nabla_{e_i}e_i}\omega.
\end{equation}
For every left-invariant vector field $X$, the Koszul formula gives $\langle\sum_i\nabla_{e_i}e_i,X\rangle=\sum_i\langle[X,e_i],e_i\rangle=\tr\ad_X$, so $\sum_i\nabla_{e_i}e_i=U$. The geometric identity $\delta\omega=-\sum_i\iota_{e_i}\nabla_{e_i}\omega$ now turns \eqref{eq:codifferential-adjoint} into \eqref{eq:codifferential}.

In particular, the second term in \eqref{eq:codifferential} must be retained in a non-unimodular Lie algebra. We put $\Delta=-(d\delta+\delta d)$. Thus the torsion equation in \eqref{eq:soliton} is equivalent to $\Delta H=-2\lambda H+\rho(D)H$, where
\begin{equation}\label{eq:form-action}
(\rho(A)\omega)(X_1,\ldots,X_p)
=-\sum_{j=1}^p\omega(X_1,\ldots,AX_j,\ldots,X_p).
\end{equation}
The three-form contraction in \eqref{eq:soliton} satisfies $\tr H^2=|H|^2=6\|H\|_\wedge^2$. For a two-form $\alpha$, the associated skew endomorphism $K_\alpha$ is defined by $\langle K_\alpha X,Y\rangle=\alpha(X,Y)$. Then $|\alpha|^2=\|K_\alpha\|^2$.

The correspondence between algebraic generalized Ricci solitons and self-similar solutions is established in \cite[Proposition~3.4 and Lemma~7.5]{FuLaSt2024}. We write the construction explicitly in our conventions. Let $G$ be the simply connected group with Lie algebra $\g$. Set $a(t)=1-2\lambda t$ and $s(t)=\int_0^t a(u)^{-1}\,du$ on the interval where $a(t)>0$. Let $\Phi_t\in\Aut(G)$ have differential $d_e\Phi_t=\exp(-s(t)D)$. The associated self-similar family is
\begin{equation}\label{eq:self-similar-family}
g_t=a(t)\Phi_t^*g,\qquad H_t=a(t)\Phi_t^*H.
\end{equation}
Indeed, symmetry of $D$ and the definition of $\rho$ give
\begin{equation}\label{eq:self-similar-derivatives}
\begin{aligned}
\left.\frac{d}{dt}\right|_{t=0}g_t
&=-2\lambda g-2g(D\,\cdot,\cdot)=-2\Ric_g+\tfrac12H_g^2,\\
\left.\frac{d}{dt}\right|_{t=0}H_t
&=-2\lambda H+\rho(D)H=-d\delta_gH.
\end{aligned}
\end{equation}
The tensor $\Ric_g$ and the two-tensor $H_g^2$ are unchanged by simultaneous scaling $(g,H)\mapsto(ag,aH)$, and $\delta_{ag}(aH)=\delta_gH$. Naturality under pullback and $s'(t)=a(t)^{-1}$ therefore extend \eqref{eq:self-similar-derivatives} to every time in \eqref{eq:self-similar-family}. This verifies \eqref{eq:intro-grf} directly and explains why the same scaling factor is used for $g$ and $H$.

We record the curvature conventions used on both the base algebra and the double. For a bracket $\mu$ on a Euclidean vector space, set
\begin{equation}\label{eq:moment-operator}
\langle M_\mu X,Y\rangle
=-\tfrac12\sum_i\langle\mu(X,e_i),\mu(Y,e_i)\rangle
+\tfrac14\sum_{i,j}\langle\mu(e_i,e_j),X\rangle\langle\mu(e_i,e_j),Y\rangle.
\end{equation}
If $\mu$ is a Lie bracket and $B$ is its Killing operator, the classical Ricci formula \cite[Corollary~7.38]{Be1987} gives
\begin{equation}\label{eq:ricci-decomposition}
\Ric=M_\mu-\tfrac12B-\Sym(\ad_U).
\end{equation}
The moment-map interpretation of $M_\mu$ goes back to Lauret \cite[Section~3]{La2003}. The infinitesimal action on brackets is $\Theta(A)\mu=A\mu-\mu(A\,\cdot,\cdot)-\mu(\cdot,A\,\cdot)$. With the full bracket norm $\|\mu\|^2=\sum_{i,j}\|\mu(e_i,e_j)\|^2$, we use the normalization in \cite[equation~(15)]{LaLa2014},
\begin{equation}\label{eq:moment-identity}
4\tr(M_\mu A)=\langle\Theta(A)\mu,\mu\rangle,\qquad
\tr M_\mu=-\tfrac14\|\mu\|^2.
\end{equation}
The first identity holds for every endomorphism $A$, since the skew-adjoint part acts skew-adjointly on bracket tensors. The second follows by taking $A=I$. In particular, $\tr(M_\mu A)=0$ for every derivation, whether or not it is symmetric, as observed in \cite[Remark~2.4]{LaLa2014}.

The next proposition applies Milnor's curvature identity for vectors orthogonal to the derived algebra \cite[Lemma~2.3 and its proof, pp.~313--314]{Mi1976} to the mean-curvature vector $U$.

\begin{proposition}\label{prop:nonunimodular-expanding}
Every non-unimodular algebraic generalized Ricci soliton is expanding.
\end{proposition}

\begin{proof}
Symmetry of $D$ and the derivation identity give $\langle DU,X\rangle=\tr\ad_{DX}=\tr[D,\ad_X]=0$. Also $U\perp[\g,\g]$. In \eqref{eq:moment-operator}, the bracket-output term therefore vanishes at $U$, and the mean-curvature term in \eqref{eq:ricci-decomposition} vanishes because $[U,U]=0$. Thus $\langle\Ric U,U\rangle=-\frac12\|\ad_U\|^2-\frac12\tr((\ad_U)^2)=-\|\Sym(\ad_U)\|^2$. Pairing the metric equation with $U\otimes U$ yields
\begin{equation}\label{eq:nonunimodular-sign}
\lambda|U|^2=-\|\Sym(\ad_U)\|^2-\tfrac14|\iota_UH|^2.
\end{equation}
For a non-unimodular algebra, $U\ne0$ and $\tr\ad_U=|U|^2>0$. Thus $\Sym(\ad_U)\ne0$, and \eqref{eq:nonunimodular-sign} gives $\lambda<0$.
\end{proof}

The generalized scalar curvature detects the steady case without any structural assumption on $\g$. We give a direct algebraic proof. The identity also follows from the invariant scalar evolution in \cite[Section~6, equation~(44)]{FuLaSt2024}.

\begin{lemma}\label{lem:invariant-scalar}
Every algebraic generalized Ricci soliton satisfies
\begin{equation}\label{eq:invariant-scalar}
2\lambda\left(\scal-\frac1{12}|H|^2\right)
=2\|R\|^2+\frac12|\delta H|^2.
\end{equation}
Consequently, every steady algebraic generalized Ricci soliton satisfies $R=0$, $D=0$ and $\delta H=0$.
\end{lemma}

\begin{proof}
Set $V=\ad_U$, $S=\Sym V$, $h=\iota_UH$ and $\eta=\delta H$. For every derivation $A$, automorphism invariance of the Killing form gives $A^*B+BA=0$, hence $\tr(BA)=0$. Together with \eqref{eq:moment-identity}, this gives
\begin{equation}\label{eq:static-derivation-pairings}
\tr(\Ric D)=-\tr(SD)=-\tr(VD),
\qquad
\tr(\Ric V)=-\|S\|^2.
\end{equation}
Here $D=D^*$ and $V$ are derivations. Since $dH=0$, Cartan's formula gives $dh=\rho(V)H$. The full-contraction pairing and \eqref{eq:codifferential} therefore imply
\begin{equation}\label{eq:static-cartan-pairing}
-3\tr(H^2V)=\langle\rho(V)H,H\rangle
=\langle dh,H\rangle
=3\langle h,\eta-h\rangle.
\end{equation}
Pairing $\Ric-H^2/4=\lambda I+D$ with $V$ and using $\tr V=|U|^2$, we obtain
\begin{equation}\label{eq:static-mean-pairing}
\lambda|U|^2+\tr(DV)
=-\|S\|^2+\frac14\langle h,\eta\rangle-\frac14|h|^2.
\end{equation}
The mean-curvature identity \eqref{eq:nonunimodular-sign} now gives $\tr(DV)=\langle h,\eta\rangle/4$. Consequently,
\begin{equation}\label{eq:static-modular-cancellation}
4\tr(\Ric D)+\langle h,\eta\rangle=0.
\end{equation}
No unimodularity assumption has been used.

Next pair $d\eta=2\lambda H-\rho(D)H$ with $H$. Since $\langle\rho(D)H,H\rangle=-3\tr(H^2D)$, the degree factor in \eqref{eq:norm-convention} yields
\begin{equation}\label{eq:static-torsion-pairing}
\tr(H^2D)=|\eta|^2-\langle h,\eta\rangle
-\frac{2\lambda}{3}|H|^2.
\end{equation}
Finally, $R=\lambda I+D$ gives
\begin{equation}\label{eq:static-final-pairing}
\begin{aligned}
\|R\|^2
&=\lambda\tr R+\tr(\Ric D)-\frac14\tr(H^2D)\\
&=\lambda\left(\scal-\frac1{12}|H|^2\right)-\frac14|\eta|^2,
\end{aligned}
\end{equation}
where \eqref{eq:static-modular-cancellation} cancels the two modular terms. This proves \eqref{eq:invariant-scalar}. If $\lambda=0$, both nonnegative terms vanish, and $R=\lambda I+D$ gives $D=0$.
\end{proof}

Let $\mathfrak r$ be the radical and $\n$ the nilradical of $\g$. Set $\af=\n^\perp$, $m=\dim\n$, $r=\dim\af$ and $N=m+r$. At this point $\af$ is only a vector-space complement. We use the classical identities of \cite[Lemma~2.6 and Remark~2.7]{LaLa2014}, in the case of trivial isotropy. The proof below specifies why the induced derivation vanishes on the central summand of $\g/\n$. 

\begin{lemma}\label{lem:general-nilradical}
Every derivation $A$ of $\g$ satisfies $A\mathfrak r\subset\n$. Moreover, $[\g,\mathfrak r]\subset\n$, the quotient $\g/\n$ is reductive, and the derivation induced by $A$ on this quotient has zero trace. In particular, a symmetric derivation preserves both $\n$ and $\af$ and has trace zero on $\af$. Also $B\n=0$ and $U\in\af$.
\end{lemma}

\begin{proof}
For a solvable Lie algebra $\mathfrak t$ and $A\in\Der(\mathfrak t)$, the semidirect extension $\R z\ltimes_A\mathfrak t$ is solvable. Lie's theorem and Engel's theorem show that its derived algebra is a nilpotent ideal. Its intersection with $\mathfrak t$ is a nilpotent ideal containing $A\mathfrak t$. Thus every derivation of $\mathfrak t$ maps into its nilradical.

The nilradical of $\mathfrak r$ equals $\n$. Indeed, it is characteristic in $\mathfrak r$ and therefore a nilpotent ideal of $\g$, while $\n\subset\mathfrak r$ gives the reverse inclusion. Restricting $A$ and $\ad_X$ to $\mathfrak r$ gives $A\mathfrak r\subset\n$ and $[\g,\mathfrak r]\subset\n$. Levi decomposition identifies $\g/\n$ with the direct sum of a semisimple algebra and the center $\mathfrak r/\n$. The induced derivation vanishes on the center and is inner on the semisimple summand, so its trace is zero.

The nilradical is characteristic. Symmetry of $A$ therefore makes $\af$ invariant as well. For $Z\in\n$, the operator $\ad_Z$ induces zero on $\g/\n$ and strictly lowers the lower central filtration of $\n$. Every $\ad_X$ preserves this filtration. Hence $\tr\ad_Z=\tr(\ad_Z\ad_X)=0$, which proves the assertions about $U$ and $B$.
\end{proof}

Automorphisms preserve the Killing form. Differentiating this invariance gives $A^*B+BA=0$ for every $A\in\Der(\g)$. Since $B=B^*$, taking traces gives
\begin{equation}\label{eq:killing-derivation-trace}
\tr(BA)=0\qquad(A\in\Der(\g)).
\end{equation}
This identity will cancel the Killing term in the global square without a sign assumption on $B$.

We record separately the stronger conclusions available for solvable algebras.

\begin{lemma}\label{lem:solvable-preliminaries}
For every solvable Lie algebra,
\begin{equation}\label{eq:solvable-preliminaries}
[\g,\g]\subset\n,\qquad B\n=0,\qquad U\in\af,\qquad
A\g\subset\n\quad(A\in\Der(\g)).
\end{equation}
Consequently, a symmetric derivation $D$ vanishes on $\af$, and $dX^\flat=0$ for every $X\in\af$.
\end{lemma}

\begin{proof}
Here $\mathfrak r=\g$, so Lemma~\ref{lem:general-nilradical} gives every assertion in \eqref{eq:solvable-preliminaries}. If $D=D^*$ and $D\g\subset\n$, symmetry gives $D|_\af=0$. Finally, $X\perp[\g,\g]$ for $X\in\af$, so $dX^\flat=0$.
\end{proof}

\begin{lemma}\label{lem:horizontal}
If \eqref{eq:soliton} holds on a solvable Lie algebra, then
\begin{equation}\label{eq:horizontal}
\lambda\|X\|^2=-\|\Sym(\ad_X)\|^2-\tfrac14|\iota_XH|^2
\quad\text{for every }X\in\af.
\end{equation}
\end{lemma}

\begin{proof}
For $X\in\af$, the positive term in \eqref{eq:moment-operator} vanishes because $[\g,\g]\subset\n$. Also $\langle[U,X],X\rangle=0$. Formula \eqref{eq:ricci-decomposition} gives $\Ric(X,X)=-\tfrac12\|\ad_X\|^2-\tfrac12\tr((\ad_X)^2)=-\|\Sym(\ad_X)\|^2$. Combine this with $DX=0$ and the metric equation in \eqref{eq:soliton}.
\end{proof}

We recall the harmonicity criterion of \cite[Theorem~3.3]{ChChZh2026b}, which also applies to non-unimodular Lie algebras. If $\eta=\delta H$ for a solution of \eqref{eq:soliton}, then
\begin{equation}\label{eq:killing-criterion}
\tr(K_\eta^*K_\eta B)\ge0\quad\Longrightarrow\quad\eta=0.
\end{equation}

For a nilpotent algebra, $B=0$, so \eqref{eq:killing-criterion} proves harmonicity for every soliton constant. The structural assertion in Theorem~\ref{thm:structure} is vacuous in this case. For a solvable nonnilpotent algebra, \eqref{eq:horizontal} gives $\lambda\le0$. Thus the solvable harmonicity assertion follows from the general nonshrinking assertion. The steady structural assertion will be proved in Section~\ref{sec:structure}, using the equality case of \eqref{eq:horizontal}.

Compact homogeneous steady generalized Ricci solitons are Bismut--Ricci-flat by \cite[Corollary~2.4]{PoRa2025}. The proof uses a gradient representative and integration of the trace equation. For compact semisimple Lie groups, \cite[Theorem~7.9]{FuLaSt2026} proves the same conclusion for semi-algebraic generalized Ricci solitons with $\lambda\le0$ by an algebraic trace identity, without assuming that the metric is bi-invariant. Under the bi-invariant metric assumption, the following Weitzenb\"ock argument also treats the shrinking case.

\begin{proposition}\label{prop:bi-invariant-harmonicity}
Let $\g$ be compact semisimple and let its metric be bi-invariant. Every algebraic generalized Ricci soliton on this metric has harmonic torsion. If $\lambda>0$, then $H=0$.
\end{proposition}

\begin{proof}
Every derivation of $\g$ is inner, and bi-invariance makes every inner derivation skew-adjoint. Since $D=D^*$, we have $D=0$. In this proof, all form inner products are exterior inner products. For $\alpha\in\Lambda^2\g^*$, define $b\alpha=\sum_{i<j}\alpha(e_i,e_j)[e_i,e_j]$. The Levi-Civita connection is $\nabla_XY=[X,Y]/2$, which gives
\[
(\delta\alpha)(X)=\tfrac12\sum_i\alpha(e_i,[e_i,X])=-\langle b\alpha,X\rangle.
\]
The curvature operator on two-forms is $\mathcal R=b^*b/4$. In our sign convention, the two-form Weitzenb\"ock formula \cite[Chapter~1, Section~I, equation~(1.154)]{Be1987} reads
\begin{equation}\label{eq:bi-invariant-weitzenbock}
(d\delta+\delta d+\rho(\Ric))\alpha
=\nabla^*\nabla\alpha-2\mathcal R\alpha,
\qquad \mathcal R=\tfrac14b^*b.
\end{equation}
Here $\rho(\Ric)$ is the negative of the usual Ricci action on two-forms.

Set $\eta=\delta H$. Since $D=0$, the torsion equation gives $d\eta=2\lambda H$, and hence $\delta d\eta=2\lambda\eta$. Also $\delta\eta=0$, so $b\eta=0$ and $\mathcal R\eta=0$. The metric equation gives $\rho(\Ric)\eta=-2\lambda\eta+\rho(H^2)\eta/4$. Apply \eqref{eq:bi-invariant-weitzenbock} to $\eta$, take its exterior inner product with $\eta$, and integrate over a compact integrating group. Dividing by its volume and using $\langle\rho(H^2)\eta,\eta\rangle_\wedge=-\tr(K_\eta^*K_\eta H^2)$ gives
\begin{equation}\label{eq:bi-invariant-defect}
0=\|\nabla\eta\|_\wedge^2
+\tfrac14\tr(K_\eta^*K_\eta H^2).
\end{equation}
Both terms in \eqref{eq:bi-invariant-defect} are nonnegative. Hence $\nabla\eta=0$ and $d\eta=0$. Unimodularity gives $\delta=d^\dagger$, and therefore $\|\eta\|_\wedge^2=\langle H,d\eta\rangle_\wedge=0$. Finally, with $D=0$ the torsion equation reads $d\delta H=2\lambda H$. If $\lambda>0$, harmonicity gives $H=0$.
\end{proof}

The steady conclusion of Theorem~\ref{thm:harmonicity} does not imply flatness on a general Lie algebra. On a nonzero compact semisimple algebra with bi-invariant metric, the Cartan form $H(X,Y,Z)=\langle[X,Y],Z\rangle$ is parallel and satisfies $H^2=-B=4\Ric$. This is the standard Bismut-flat construction of \cite[Proposition~3.53]{GaSt2021}. Thus $\lambda=0$ and $D=0$ give a nonflat steady solution with nonzero torsion. The solvability assumption in the flatness conclusion of Corollary~\ref{cor:expanding-solvable} is therefore essential.

For an arbitrary left-invariant metric on a compact semisimple Lie algebra, a symmetric derivation is still zero, since every derivation is inner and has purely imaginary spectrum. However, the curvature-operator argument above uses bi-invariance. Thus the equations $\Ric-H^2/4=\lambda I$ and $d\delta H=2\lambda H$ for $\lambda>0$ remain outside the preceding proof when this metric symmetry is absent.

In Sections~\ref{sec:cotangent}--\ref{sec:global}, $\g$ is an arbitrary metric Lie algebra and we assume only
\begin{equation}\label{eq:expanding-reduction}
\lambda=-c<0.
\end{equation}
Both $r=0$ and $m=0$ are permitted. The latter occurs when $\g$ is semisimple.

\section{Nilpotent quadratic quotients and a positive-definite operator}\label{sec:cotangent}\label{sec:quadratic-git}
In this section, we verify the curvature formula on the twisted cotangent Lie algebra, construct its nilpotent quadratic quotient and obtain the positive-definite operator needed for the global identity.

We use the twisted cotangent construction of \cite[Section~3]{Bo1997}, in the form recalled in \cite[Lemma~4.6]{FuLaSt2024}. It turns the torsion equation into a derivation equation on $\h=\g\oplus\g^*$. On $\h$, use the Euclidean product metric and the neutral form
\begin{equation}\label{eq:neutral-form}
q(X+\xi,Y+\zeta)=\xi(Y)+\zeta(X)=\langle J(X+\xi),Y+\zeta\rangle,
\qquad J=\begin{pmatrix}0&I\\I&0\end{pmatrix}.
\end{equation}
Musical identifications in this formula and in subsequent block matrices are induced by the metric on $\g$. In particular, $J=J^*=J^{-1}$. Our form $q$ is twice the neutral pairing in \cite{FuLaSt2024,FuLaSt2026}, which does not change its orthogonal group. The positive product metric agrees with the metric used in the curvature calculation of \cite[Section~3]{FuLaSt2026}. Define
\begin{equation}\label{eq:cotangent-bracket}
[X+\xi,Y+\zeta]_\nu=[X,Y]+\ad_X^{\mathrm{coad}}\zeta-\ad_Y^{\mathrm{coad}}\xi+H(X,Y,\cdot),
\end{equation}
where $\ad_X^{\mathrm{coad}}\xi=-\xi\circ\ad_X$. We reserve the superscript $*$ for the Euclidean adjoint. Closedness of $H$ is exactly the additional Jacobi identity in \eqref{eq:cotangent-bracket}, and $q$ is invariant. Thus $\nu$ is a quadratic Lie bracket.

Write $\mathfrak{so}(q)=\{A:A^*J+JA=0\}$. This Lie algebra is closed under Euclidean transpose, and every inner derivation of $\nu$ belongs to it. The diagonal blocks of $\ad_{X+\xi}^\nu$ are $\ad_X$ and $-(\ad_X)^*$, so $\nu$ is unimodular. Taking the trace of a product of two such adjoints gives the Killing operator, as in \cite[equation~(22)]{FuLaSt2026},
\begin{equation}\label{eq:double-killing}
B_\nu=\begin{pmatrix}2B&0\\0&0\end{pmatrix}.
\end{equation}
For symmetric $A$, set $P_q(A)=(A-JAJ)/2$, the orthogonal projection to $\operatorname{Sym}(\h)\cap\mathfrak{so}(q)$.

We use the curvature formula of \cite[Theorem~3.1, equation~(18)]{FuLaSt2026}, with $\Ric_\nu=M_\nu-B_\nu/2$. The vector $U$ below is the mean-curvature vector of the base algebra. We give the block calculation to fix the normalization of the off-diagonal terms.

\begin{proposition}[{\cite[Theorem~3.1]{FuLaSt2026}}]\label{prop:cotangent-curvature}
The generalized Ricci operator has the following block form and moment expression, with $\eta=\delta H$,
\begin{equation}\label{eq:curvature}
P=\begin{pmatrix}R&\tfrac12K_\eta\\-\tfrac12K_\eta&-R\end{pmatrix}
=P_q\left(\tfrac23M_\nu-\tfrac12B_\nu\right)-\Sym(\ad_U^\nu).
\end{equation}
\end{proposition}

\begin{proof}
Let $e=d^\dagger H$, $h=\iota_UH$ and $S=\Sym(\ad_U)$. Thus $\eta=e+h$. The bracket \eqref{eq:cotangent-bracket} has one copy of the original bracket, two copies of its coadjoint coefficients, and the three-form component. This gives the norm identity in \cite[equation~(40)]{FuLaSt2024},
\begin{equation}\label{eq:double-norm}
\|\nu\|^2=3\|\mu\|^2+|H|^2.
\end{equation}
To check the diagonal moment blocks, put $A_\mu(X,Y)=\sum_i\langle\mu(X,e_i),\mu(Y,e_i)\rangle$ and $O_\mu(X,Y)=\sum_{i,j}\langle\mu(e_i,e_j),X\rangle\langle\mu(e_i,e_j),Y\rangle$. Here $A_\mu$ and $O_\mu$ denote symmetric operators through the metric, so $M_\mu=-A_\mu/2+O_\mu/4$. Substitution in \eqref{eq:moment-operator} gives
\begin{equation}\label{eq:double-diagonal-moment}
(M_\nu)_{\g\g}=-A_\mu+\tfrac14O_\mu-\tfrac12H^2,\qquad
(M_\nu)_{\g^*\g^*}=\tfrac12A_\mu-\tfrac12O_\mu+\tfrac14H^2.
\end{equation}
To compute the off-diagonal blocks, write $\mu_{ij}^{\ell}=\langle\mu(e_i,e_j),e_\ell\rangle$ and $H_{kij}=H(e_k,e_i,e_j)$. Set $G_{\ell k}=\sum_{i,j}\mu_{ij}^{\ell}H_{kij}$. Direct contraction in \eqref{eq:moment-operator} and the exterior adjoint identity give
\begin{equation}\label{eq:double-mixed-contractions}
N_\nu:=(M_\nu)_{\g^*\g}=-\tfrac12G+\tfrac14G^*,\qquad
K_{d^\dagger H}=\tfrac12(G-G^*).
\end{equation}
Here $*$ is matrix transpose in the chosen orthonormal basis. Since $M_\nu$ is symmetric, its opposite block is $N_\nu^*$. Equation \eqref{eq:double-mixed-contractions} therefore yields
\begin{equation}\label{eq:double-off-diagonal-projection}
(P_q(M_\nu))_{\g^*\g}
=\tfrac12(N_\nu-N_\nu^*)
=-\tfrac38(G-G^*)
=-\tfrac34K_{d^\dagger H}.
\end{equation}
Combining \eqref{eq:double-diagonal-moment} and \eqref{eq:double-off-diagonal-projection}, we obtain
\begin{equation}\label{eq:projected-double-moment}
\frac23P_q(M_\nu)=
\begin{pmatrix}
M_\mu-\tfrac14H^2&\tfrac12K_e\\
-\tfrac12K_e&-M_\mu+\tfrac14H^2
\end{pmatrix}.
\end{equation}
The adjoint $d^\dagger$ in \eqref{eq:double-mixed-contractions} is taken for the standard exterior inner products. With full contractions, its defining pairing is $\langle d^\dagger H,\alpha\rangle=\langle H,d\alpha\rangle/3$ for $\alpha\in\Lambda^2\g^*$. Thus the degree factor in \eqref{eq:norm-convention} is retained in the off-diagonal calculation.

Equation \eqref{eq:double-killing} gives $P_q(B_\nu)=\operatorname{diag}(B,-B)$. The adjoint formula in \cite[equation~(23)]{FuLaSt2026}, or directly \eqref{eq:cotangent-bracket}, gives
\begin{equation}\label{eq:double-modular-blocks}
\Sym(\ad_U^\nu)=
\begin{pmatrix}S&-\tfrac12K_h\\\tfrac12K_h&-S\end{pmatrix}.
\end{equation}
Combining \eqref{eq:projected-double-moment}, \eqref{eq:double-modular-blocks} and \eqref{eq:ricci-decomposition} gives exactly \eqref{eq:curvature}. In particular, the off-diagonal form is $e+h=\delta H$. Replacing it by $d^\dagger H$ would lose the geometric correction used in the global identity.
\end{proof}

\begin{proposition}[{\cite[Proposition~7.4 and Lemma~7.6]{FuLaSt2024}}]\label{prop:cotangent-derivation}
Under \eqref{eq:expanding-reduction}, the endomorphism
\begin{equation}\label{eq:double-derivation}
L=\begin{pmatrix}D&0\\-K_\eta&2cI-D\end{pmatrix}
\end{equation}
is a derivation of $\nu$. Moreover, $L-cI\in\mathfrak{so}(q)$ and $P=\Sym L-cI$.
\end{proposition}

\begin{proof}
The base and mixed components of the derivation identity follow from $D\in\Der(\g)$. To check the remaining component, replace $-K_\eta$ in \eqref{eq:double-derivation} by $K_b$, where $b$ is any two-form, and denote the resulting operator by $L_b$. For $X,Y,Z\in\g$, the covector component of its derivation defect is
\begin{equation}\label{eq:double-derivation-defect}
\begin{aligned}
&\bigl(\pr_{\g^*}\{L_b\nu(X,Y)-\nu(L_bX,Y)-\nu(X,L_bY)\}\bigr)(Z)\\
&\quad=b([X,Y],Z)-b(X,[Y,Z])+b(Y,[X,Z])\\
&\qquad\quad+2cH(X,Y,Z)-H(DX,Y,Z)-H(X,DY,Z)-H(X,Y,DZ)\\
&\quad=\bigl(-db+2cH+\rho(D)H\bigr)(X,Y,Z).
\end{aligned}
\end{equation}
For $b=-\eta$, equation~\eqref{eq:double-derivation-defect} vanishes by the torsion equation in \eqref{eq:soliton}, since $\lambda=-c$. Direct multiplication gives $L^*J+JL=2cJ$. Taking the symmetric part in \eqref{eq:double-derivation} and using $R=-cI+D$ proves the last assertion.
\end{proof}

The lower-left block of $L$ is retained throughout the argument. In particular, Proposition~\ref{prop:cotangent-derivation} does not assume that $L$ is symmetric.

Identify $\af^*\subset\g^*$ with the annihilator of $\n$. The subspace $I=\n\oplus\g^*$ is an ideal of $\h$, and $I^{\perp_q}=\af^*\subset I$. The orthogonal complement of an ideal in a quadratic Lie algebra is again an ideal \cite[Proposition~2.1(i)]{Bo1997}. Thus $\af^*$ is an ideal of $\h$. It is central in $I$, since $[\n,\g]\subset\n$, but need not be central in $\h$. The radical of the restriction of $q$ to $I$ is exactly $\af^*$.

\begin{proposition}\label{prop:nilpotent-quotient}
The quotient
\begin{equation}\label{eq:nilpotent-quotient}
W=(\n\oplus\g^*)/\af^*\simeq\n\oplus\n^*
\end{equation}
has an induced nilpotent quadratic bracket $\nu_0$. Equip $W$ with the quotient Euclidean metric. This metric and the induced neutral form $q_W$ identify with the standard forms on $\n\oplus\n^*$. The derivation $L$ and the inner derivations $\ad_A^\nu$, $A\in\af$, induce derivations of $\nu_0$.
\end{proposition}

\begin{proof}
The two ideal inclusions give the quotient bracket, and $q$ induces a nondegenerate invariant form by \cite[Proposition~2.3(iii)]{Bo1997}. The Euclidean orthogonal complement of $\af^*$ in $I$ is $\n\oplus\n^*$, which gives the asserted metric identification. Under \eqref{eq:nilpotent-quotient}, the bracket is the cotangent bracket of $\n$ twisted by the closed form $\gamma=H|_{\Lambda^3\n}$. Its nilpotency is a special case of \cite[Theorem~3.1(i)]{Bo1997}. For completeness, the adjoint of any $X+\xi\in\n\oplus\n^*$ is block triangular, with diagonal blocks $\ad_X|_\n$ and $-(\ad_X|_\n)^*$. Both are nilpotent, so the full adjoint is nilpotent. Engel's theorem proves the assertion.

The operator $L$ preserves both ideals defining the quotient because $D$ preserves $\n$ and $\af$ by Lemma~\ref{lem:general-nilradical}. Inner derivations preserve these ideals as well. They therefore induce derivations on $W$.
\end{proof}

Choose an orthonormal basis $a_1,\ldots,a_r$ of $\af$. On $W$, write
\begin{equation}\label{eq:quotient-operators}
T_i=\begin{pmatrix}A_i&0\\K_{h_i}&-A_i^*\end{pmatrix},\qquad
L_0=\begin{pmatrix}D_\n&0\\-K_{\eta|_\n}&2cI-D_\n\end{pmatrix},
\end{equation}
where $A_i=\ad_{a_i}|_\n$ and $h_i=H(a_i,\cdot,\cdot)|_{\n\times\n}$. If $J_W$ denotes the induced neutral involution, then
\begin{equation}\label{eq:quotient-ray}
T_i\in\Der(\nu_0)\cap\mathfrak{so}(q_W),\qquad
L_0-cI\in\mathfrak{so}(q_W),\qquad
\Theta(L_0-cI)\nu_0=c\nu_0.
\end{equation}
The last identity follows from $\Theta(L_0)\nu_0=0$ and $\Theta(I)\nu_0=-\nu_0$. Cartan's formula and $dH=0$ give $d_\n h_i=\rho(A_i)\gamma$, where $d_\n$ is the Chevalley--Eilenberg differential on $\n$. Thus the lower-left block of $T_i$ accounts for the possible failure of $\gamma$ to be invariant under $A_i$.

\begin{remark}\label{rem:quotient-not-subalgebra}
The identification in \eqref{eq:nilpotent-quotient} is an identification of a quotient with a Euclidean representative. Before eliminating the components $H(\af,\g,\g)$, the representative $\n\oplus\n^*$ need not be a subalgebra of $\h$, since its bracket may have an $\af^*$ component. All uses of $\nu_0$, $T_i$ and $L_0$ refer to the quotient.
\end{remark}

The argument below supplies the positive operator used in the global trace identity. We apply the real geometric invariant theory of \cite{BoLa2020} to the orthogonal group of the invariant neutral form. For the general linear action on nilpotent brackets, the corresponding positivity statement is \cite[Lemma~2.17]{La2010}. We adapt its proof to the representation used here.

Let $W$ be a real Euclidean vector space, and let $J=J^*=J^{-1}$ define a nondegenerate symmetric form $q(X,Y)=\langle JX,Y\rangle$. Set $G=\mathrm O(q)$, $K=G\cap\mathrm O(W)$ and
\begin{equation}\label{eq:quadratic-cartan-decomposition}
\mathfrak g_q=\mathfrak{so}(q),\qquad
\mathfrak p_q=\{A\in\operatorname{Sym}(W):AJ+JA=0\}.
\end{equation}
We equip $\mathfrak g_q$ with the Hilbert--Schmidt inner product. The space
\begin{equation}\label{eq:quadratic-bracket-representation}
\mathcal V_q=\{\mu\in\Lambda^2W^*\otimes W:q(\mu(\cdot,\cdot),\cdot)\in\Lambda^3W^*\}
\end{equation}
is $G$-invariant under change of basis. Its infinitesimal action is
\begin{equation}\label{eq:quadratic-theta-action}
(\Theta(A)\mu)(X,Y)=A\mu(X,Y)-\mu(AX,Y)-\mu(X,AY).
\end{equation}
The identification with $\Lambda^3W^*$ given by $q$ intertwines this action with the covariant tensor action. All tensor norms in this section are full contraction norms. In particular, $\Theta(A)^*=\Theta(A^*)$ for $A\in\mathfrak g_q$.

Following the normalization of \cite[Definition~1.2]{BoLa2020}, define the moment map $m_q:\mathcal V_q\setminus\{0\}\longrightarrow\mathfrak p_q$ by
\begin{equation}\label{eq:quadratic-moment-definition}
\langle m_q(\mu),A\rangle
=\frac{\langle\Theta(A)\mu,\mu\rangle}{\|\mu\|^2},
\qquad A\in\mathfrak p_q.
\end{equation}
If $M_\mu$ is the usual bracket moment operator, normalized by $4\operatorname{tr}(M_\mu A)=\langle\Theta(A)\mu,\mu\rangle$, then
\begin{equation}\label{eq:quadratic-moment-projection}
m_q(\mu)=\frac{4}{\|\mu\|^2}\operatorname{pr}_{\mathfrak p_q}M_\mu
=\frac{2}{\|\mu\|^2}(M_\mu-JM_\mu J).
\end{equation}
The moment-map interpretation for Dorfman brackets and the neutral orthogonal group is given in \cite[Remark~7.10]{FuLaSt2024}. Formula \eqref{eq:quadratic-moment-projection} fixes the scale-invariant normalization used here on $\mathcal V_q$.

To apply the faithful closed-image formulation of \cite{BoLa2020}, consider the representation of $G$ on $\mathcal V_q\oplus W$, with the standard action on $W$. This representation is faithful and has closed image, since its second component recovers the original matrix in $G$. It is compatible with the given inner products. We retain the Hilbert--Schmidt inner product on $\mathfrak g_q\subset\operatorname{End}(W)$, which satisfies the compatibility conditions in \cite[Section~1]{BoLa2020}. At $(\mu,0)$ the moment map is $m_q(\mu)$. The gradient of its squared norm has zero $W$-component, so the critical point equation is unchanged. The spectral slices restrict to the corresponding slices on $\mathcal V_q\oplus\{0\}$, and the orbit closures used in semistability and in the scaled critical limit remain in this invariant subspace. The following proposition therefore applies with the normalization \eqref{eq:quadratic-moment-definition}.

\begin{proposition}[{\cite[Theorem~1.3, Lemma~7.13 and Corollaries~7.15 and~8.4]{BoLa2020}}]\label{prop:quadratic-git-input}
For every $0\ne\mu\in\mathcal V_q$, there is $\beta\in\mathfrak p_q$ with the following properties. Set $s=\|\beta\|^2$, and let $V_\beta^{\ge0}$ and $V_\beta^0$ denote the nonnegative spectral subspace and the kernel of $\Theta(\beta)-sI_{\mathcal V_q}$. Let $p_\beta:V_\beta^{\ge0}\to V_\beta^0$ be the orthogonal projection. Set
\begin{equation}\label{eq:quadratic-slice-subgroup}
\mathfrak h_\beta=\{A\in\mathfrak g_q:[A,\beta]=0,\ \langle A,\beta\rangle=0\}.
\end{equation}
Let $K_\beta=\{k\in K:k\beta=\beta k\}$ and $H_\beta=K_\beta\exp(\mathfrak h_\beta\cap\mathfrak p_q)$. Set $U_\beta=\{v\in V_\beta^{\ge0}:0\notin\overline{H_\beta\cdot p_\beta(v)}\}$. The associated parabolic subgroup is
\begin{equation}\label{eq:quadratic-parabolic}
Q_\beta=\left\{g\in G:\lim_{t\to\infty}\exp(-t\beta)g\exp(t\beta)\text{ exists in }G\right\}.
\end{equation}
Its Lie algebra $\mathfrak q_\beta$ is the sum of the nonnegative eigenspaces of $\operatorname{ad}\beta$ on $\mathfrak g_q$, by \cite[Lemmas~B.2--B.3]{BoLa2020}. Then the following assertions hold.
\begin{enumerate}
\item The current tensor $\mu$ belongs to $U_\beta$. In particular, $\mu_{\min}=p_\beta(\mu)$ is nonzero and $H_\beta$-semistable. The projection formula in \cite[Lemma~7.5]{BoLa2020} gives
\begin{equation}\label{eq:quadratic-lowest-projection}
\mu_{\min}=\lim_{t\to\infty}e^{st}\exp(-t\beta)\cdot\mu.
\end{equation}
\item There is a nonzero critical point $\mu_c\in\overline{\mathbb R_{>0}G\cdot\mu}$ of $\|m_q\|^2$ such that $m_q(\mu_c)=\beta$.
\item The set $U_\beta$ is a cone. For $v\in U_\beta$ and $g\in G$, one has $g\cdot v\in U_\beta$ if and only if $g\in Q_\beta$.
\end{enumerate}
\end{proposition}

The label may be chosen for the current Euclidean metric. Indeed, the cited slice theorem first places $\mu$ in a slice by an element of $K$. Conjugating both the label and its auxiliary critical representative back by this element gives the stated formulation.

The slice inequality and the derivation assertions below are consequences of \cite[Lemma~9.1 and Corollary~9.2]{BoLa2020}. We include a proof that also gives the trace identity for an endomorphism acting by a scalar on $\mu$.

\begin{lemma}\label{lem:quadratic-git-inequalities}
Let $0\ne\mu\in\mathcal V_q$, and choose $\beta$ as above. Then
\begin{equation}\label{eq:quadratic-slice-inequality}
\langle(\Theta(\beta)-sI)\mu,\mu\rangle\ge0.
\end{equation}
If $T\in\mathfrak{so}(q)\cap\operatorname{Der}(\mu)$, then
\begin{equation}\label{eq:quadratic-stabilizer-inequalities}
\langle[\beta,T],T\rangle\ge0,\qquad
\operatorname{tr}(\beta T)=0.
\end{equation}
Equality in the first inequality holds exactly when $[\beta,T]=0$. More generally, if $A\in\mathfrak{so}(q)$ and $\Theta(A)\mu=a\mu$, then
\begin{equation}\label{eq:quadratic-ray-trace}
\operatorname{tr}(\beta A)=a.
\end{equation}
\end{lemma}

\begin{proof}
The spectral condition gives \eqref{eq:quadratic-slice-inequality}. Suppose that $\Theta(A)\mu=a\mu$. Then $\exp(tA)\cdot\mu=e^{at}\mu$ remains in the same slice. The characterization of the parabolic subgroup \eqref{eq:quadratic-parabolic} in Proposition~\ref{prop:quadratic-git-input} gives $A\in\mathfrak q_\beta$. Write $A=A_0+A_+$ according to the zero and positive eigenspaces of $\operatorname{ad}\beta$. If $[\beta,A_r]=rA_r$, then $[\Theta(\beta),\Theta(A_r)]=r\Theta(A_r)$. Thus $\Theta(A_+)$ raises eigenvalues and has zero projection to the lowest spectral subspace. Projecting the ray equation therefore gives $\Theta(A_0)\mu_{\min}=a\mu_{\min}$.

Assume first that $s>0$. Set $B=A_0-s^{-1}\langle\beta,A\rangle\beta\in\mathfrak h_\beta$. Since $\Theta(\beta)\mu_{\min}=s\mu_{\min}$, we obtain
\begin{equation}\label{eq:quadratic-ray-semistability}
\Theta(B)\mu_{\min}
=(a-\langle\beta,A\rangle)\mu_{\min}.
\end{equation}
If the coefficient on the right were nonzero, one time direction of $\exp(tB)$ would carry $\mu_{\min}$ to zero. This contradicts $H_\beta$-semistability. Therefore \eqref{eq:quadratic-ray-trace} holds. If $s=0$, the slice consists of semistable vectors for $G$ itself, and the same argument applied to $A$ gives $a=0$.

For a derivation $T$, take $a=0$. Since $T\in\mathfrak q_\beta$, write $T=\sum_{r\ge0}T_r$, where $[\beta,T_r]=rT_r$. The eigenspaces are orthogonal, and hence
\begin{equation}\label{eq:quadratic-stabilizer-square}
\langle[\beta,T],T\rangle
=\sum_{r>0}r\|T_r\|^2.
\end{equation}
This proves the remaining assertions.
\end{proof}

\begin{proposition}\label{prop:quadratic-positive-operator}
Suppose that $0\ne\mu\in\mathcal V_q$ is a nilpotent Lie bracket. Then the label in Lemma~\ref{lem:quadratic-git-inequalities} satisfies $s>0$, and
\begin{equation}\label{eq:quadratic-positive-operator}
C=\beta+sI_W>0.
\end{equation}
\end{proposition}

\begin{proof}
In fixed dimension, the Jacobi identity and the vanishing of all iterated brackets of length $\dim W+1$ are polynomial conditions. Invariance of $q$ is a linear condition. Thus nilpotent quadratic Lie brackets form a closed set preserved by positive scaling and the $G$-action. Consequently, the critical bracket $\mu_c$ is again a nonzero nilpotent quadratic Lie bracket. We prove \eqref{eq:quadratic-positive-operator} using this bracket. The operator $\beta$ is the same stratum label, so this proves positivity for the operator associated with the original bracket.

We first record an identity at $\mu_c$. If $A\in\mathfrak{so}(q)\cap\operatorname{Der}(\mu_c)$, then $[A,A^*]\in\mathfrak p_q$. Equations \eqref{eq:quadratic-moment-definition} and $\Theta(A)\mu_c=0$ give
\begin{equation}\label{eq:quadratic-critical-stabilizer}
\langle[\beta,A],A\rangle
=\langle\beta,[A,A^*]\rangle
=\frac{\|\Theta(A^*)\mu_c\|^2}{\|\mu_c\|^2}.
\end{equation}
If $\beta=0$, apply this identity to $A=\operatorname{ad}_{\mu_c}X$. It follows that $A^*$ is also a derivation. Thus $[A^*,A]=\operatorname{ad}_{\mu_c}(A^*X)$ is an inner derivation. It is symmetric and nilpotent, so it vanishes. The operator $A$ is therefore normal and nilpotent, which implies $A=0$. This holds for every $X$, contradicting $\mu_c\ne0$. Hence $s>0$.

The critical point equation is $\Theta(\beta)\mu_c=s\mu_c$ by \cite[Lemma~7.2]{BoLa2020}. Since $\Theta(I_W)\mu_c=-\mu_c$, the symmetric operator $C=\beta+sI_W$ is a derivation of $\mu_c$. Choose an eigenvector $X$ of $C$, with eigenvalue $r$, and set $A=\operatorname{ad}_{\mu_c}X$. Then $[\beta,A]=[C,A]=rA$. If $A\ne0$, equation \eqref{eq:quadratic-critical-stabilizer} gives $r\ge0$. If $r=0$, the preceding argument again gives $A=0$. Thus $r>0$ whenever $X$ is not central.

It remains to treat central eigenvectors. If $X$ is central, $q$-invariance gives $JX\perp\mu_c(W,W)$. The bracket moment formula therefore yields
\begin{equation}\label{eq:quadratic-center-moment}
\begin{aligned}
\langle M_{\mu_c}X,X\rangle
&=\frac14\sum_{i,j}\langle\mu_c(e_i,e_j),X\rangle^2\ge0,\\
\langle M_{\mu_c}JX,JX\rangle
&=-\frac12\sum_i\|\mu_c(JX,e_i)\|^2\le0.
\end{aligned}
\end{equation}
By \eqref{eq:quadratic-moment-projection}, this implies $\langle\beta X,X\rangle\ge0$. Hence $\langle CX,X\rangle\ge s\|X\|^2>0$.
Every eigenvalue of $C$ is positive, as required.
\end{proof}

\begin{corollary}\label{cor:quadratic-conformal-derivation-trace}
Under the assumptions of Proposition~\ref{prop:quadratic-positive-operator}, let $L\in\operatorname{Der}(\mu)$ satisfy $q(LX,Y)+q(X,LY)=2c\,q(X,Y)$. Then
\begin{equation}\label{eq:quadratic-conformal-trace}
\operatorname{tr}(\beta L)=c.
\end{equation}
\end{corollary}

\begin{proof}
The operator $A=L-cI_W$ belongs to $\mathfrak{so}(q)$ and satisfies $\Theta(A)\mu=c\mu$. Apply \eqref{eq:quadratic-ray-trace} and use $\operatorname{tr}\beta=0$.
\end{proof}

\begin{remark}\label{rem:quadratic-git-distinction}
Only the auxiliary bracket $\mu_c$ in the closure of the scaled $G$-orbit is assumed to be critical. Positivity transfers to the current tensor because the operator $C$ depends on the common label $\beta$. All moment maps in this argument concern the $\mathrm O(q)$ representation.
\end{remark}

\section{Harmonic torsion without a Killing sign condition}\label{sec:global}
In this section, we prove the global square identity, establish harmonicity and obtain a reductive complement to the nilradical.

We return to the expanding soliton fixed in \eqref{eq:expanding-reduction}. Apply Section~\ref{sec:quadratic-git} to the nilpotent quadratic bracket $\nu_0$ on $W$. If $\nu_0\ne0$, choose its adapted label $\beta$ and set $s=\|\beta\|^2$ and $C=\beta+sI$. If $\nu_0=0$, choose instead
\begin{equation}\label{eq:zero-bracket-choice}
\beta=0,\qquad s=1,\qquad C=I_W.
\end{equation}
Thus $s>0$ and $C>0$ in both cases. The choice \eqref{eq:zero-bracket-choice} uses no GIT stratum. In particular, $s=\|\beta\|^2$ is imposed only in the nonzero case.

In the nonzero case, the label $\beta$ is already adapted to the current tensor $\nu_0$ by Proposition~\ref{prop:quadratic-git-input}. In both cases we retain the original splitting throughout the proof. In the following block matrices, we identify $W$ with its Euclidean representative $\n\oplus\n^*$ as in Remark~\ref{rem:quotient-not-subalgebra}. Thus $\h=\af\oplus W\oplus\af^*$ is an orthogonal decomposition of vector spaces, and we define
\begin{equation}\label{eq:extended-operator}
\Ecal=\operatorname{diag}(0_\af,C,2sI_{\af^*}).
\end{equation}
Since $J$ interchanges $\af$ and $\af^*$ and $\beta J_W+J_W\beta=0$, this operator satisfies $\Ecal+J\Ecal J=2sI$. Also $\tr\beta=0$ and $\dim W=2m$, so $\tr C=2ms$ and $\|C\|^2=\|\beta\|^2+2ms^2$. Consequently,
\begin{equation}\label{eq:extended-operator-traces}
\tr\Ecal=2Ns,\qquad
\|\Ecal\|^2=\|\beta\|^2+2(m+2r)s^2.
\end{equation}
These formulas apply in both branches of the construction, including $W=0$.

Put $V=\ad_U^\nu$, $\mathcal K=L+V$ and $\Fcal=\Sym\mathcal K$. The operator $\mathcal K$ is an ordinary derivation of $\nu$. By \eqref{eq:curvature} and Proposition~\ref{prop:cotangent-derivation},
\begin{equation}\label{eq:corrected-curvature}
\Fcal=cI+P_q\left(\tfrac23M_\nu-\tfrac12B_\nu\right),\qquad
\Fcal+J\Fcal J=2cI.
\end{equation}
The inclusion of $V$ cancels exactly the mean-curvature contribution in \eqref{eq:curvature}.

The argument uses the orthogonality of the moment operator to derivations and the Killing-form cancellation in \cite[Remarks~2.4 and~2.7, Lemma~4.2]{LaLa2014}. We apply these identities to the twisted cotangent bracket and retain the block involving $\eta$ throughout the calculation.

\begin{proposition}\label{prop:global-square}
For every expanding algebraic generalized Ricci soliton on a real metric Lie algebra, the operators above satisfy
\begin{equation}\label{eq:global-square}
0=\left\|\Fcal-\frac cs\Ecal\right\|^2
+\frac{2c}{3s}\tr(M_\nu\Ecal).
\end{equation}
\end{proposition}

\begin{proof}
We first compute the pairing of $\Fcal$ and $\Ecal$. By \eqref{eq:quotient-ray}, Lemma~\ref{lem:quadratic-git-inequalities} and Corollary~\ref{cor:quadratic-conformal-derivation-trace}, the identities
\begin{equation}\label{eq:unified-label-trace}
\tr(\beta T_i)=0,\qquad
\tr(\beta L_0)=\frac cs\|\beta\|^2
\end{equation}
hold when $\nu_0\ne0$. In that case the second right-hand side is $c$. When $\nu_0=0$, we have $\beta=0$ by \eqref{eq:zero-bracket-choice}, so both identities in \eqref{eq:unified-label-trace} remain valid with the displayed normalization.

The derivation $D+\ad_U$ preserves the characteristic ideal $\n$. Let $Q=(D+\ad_U)_{\af\af}$ be its induced operator on $\g/\n$ under the orthogonal identification. Lemma~\ref{lem:general-nilradical} gives $\tr Q=0$ without a solvability assumption. The diagonal blocks of $\mathcal K$ on $\af$, $W$ and $\af^*$ are $Q$, $L_0+T_U$ and $2cI-Q^*$, respectively, where $T_U=\sum_i\langle U,a_i\rangle T_i$. Here $U\in\af$, $\tr L_0=2cm$ and $\tr T_U=0$. Equation~\eqref{eq:unified-label-trace} also gives $\tr(\beta T_U)=0$. Pairing these blocks with \eqref{eq:extended-operator}, we obtain
\begin{align}
\langle\Fcal,\Ecal\rangle
&=\tr\bigl((L_0+T_U)(\beta+sI)\bigr)+4crs\notag\\
&=\frac cs\|\beta\|^2+2c(m+2r)s
=\frac cs\|\Ecal\|^2.\label{eq:extended-operator-pairing}
\end{align}
No off-diagonal block has been set equal to zero in this calculation.

Set $Z=\Fcal-(c/s)\Ecal$. By \eqref{eq:corrected-curvature} and \eqref{eq:extended-operator}, the operator $Z$ is symmetric and belongs to $\mathfrak{so}(q)$. Hence $\tr Z=0$ and $P_q(Z)=Z$. Lemma~\ref{lem:general-nilradical} and \eqref{eq:double-killing} show that $B_\nu$ vanishes on $W\oplus\af^*$ and maps $\af$ into itself, so $\tr(B_\nu\Ecal)=0$. In the original splitting $\g\oplus\g^*$, the top-left block of $\Fcal$ is $\Sym(D+\ad_U)$. Since $D+\ad_U$ is a derivation, \eqref{eq:killing-derivation-trace} gives
\begin{equation}\label{eq:global-killing-cancellation}
\tr(B_\nu\Fcal)=2\tr\bigl(B(D+\ad_U)\bigr)=0,
\qquad \tr(B_\nu Z)=0.
\end{equation}
Thus the Killing term vanishes. Finally, $\mathcal K$ is a derivation and $M_\nu$ is symmetric. Equation \eqref{eq:moment-identity}, or \cite[Remark~2.4]{LaLa2014}, therefore gives
\begin{equation}\label{eq:moment-derivation-orthogonality}
\tr(M_\nu\Fcal)=\tr(M_\nu\mathcal K)=0.
\end{equation}
The projection $P_q$ is self-adjoint for the Hilbert--Schmidt inner product, and $P_q(Z)=Z$. Pairing \eqref{eq:corrected-curvature} with $Z$ and using $\tr Z=0$, \eqref{eq:global-killing-cancellation} and \eqref{eq:moment-derivation-orthogonality}, we obtain
\begin{equation}\label{eq:global-square-pairing}
\langle\Fcal,Z\rangle
=\tfrac23\tr(M_\nu Z)
=-\frac{2c}{3s}\tr(M_\nu\Ecal).
\end{equation}
Equation \eqref{eq:extended-operator-pairing} gives $\langle Z,\Ecal\rangle=0$, so the left-hand side of \eqref{eq:global-square-pairing} equals $\|Z\|^2$. This proves \eqref{eq:global-square}, including the zero-bracket case.
\end{proof}

We next prove that the last term of \eqref{eq:global-square} is nonnegative. The classical bracket decomposition in \cite[Lemma~2.3]{LaLa2014} separates the nilradical bracket, the action of its complement and the remaining bracket components. Here we use the Cartan three-form $\tau(X,Y,Z)=q(\nu(X,Y),Z)$ to include the torsion terms. Since $J$ is an isometry, $|\tau|^2=\|\nu\|^2$. The possible orthogonal components have the following types of arguments,
\begin{equation}\label{eq:five-cartan-types}
W^3,\qquad \af W^2,\qquad \af^2W,\qquad \af^3,\qquad \af^2\af^*.
\end{equation}
where, for example, $\af W^2$ means one argument in $\af$ and two in $W$, together with all permutations. A component with two $\af^*$ arguments vanishes by the cotangent bracket. If one $\af^*$ argument and one $W$ argument occur, the base bracket lies in $\n$ and is annihilated by $\af^*$. The last component in \eqref{eq:five-cartan-types} is $\tau(a_i,a_j,\xi)=\xi(\pr_\af[a_i,a_j])$ for $\xi\in\af^*$. It records the bracket of the reductive quotient $\g/\n$.

Define vectors in $W$ by
\begin{equation}\label{eq:horizontal-bracket-vectors}
v_{ij}=\pr_\n[a_i,a_j]+H(a_i,a_j,\cdot)|_\n\in\n\oplus\n^*.
\end{equation}
Set $H_\af=H|_{\Lambda^3\af}$. The first four Cartan components are determined by $\nu_0$, $T_i$, $v_{ij}$ and $H_\af$, respectively. Explicitly, $\tau|_{W^3}=q(\nu_0(\cdot,\cdot),\cdot)$, $\tau(a_i,w,z)=\langle J_WT_iw,z\rangle$, and $\tau(a_i,a_j,w)=\langle J_Wv_{ij},w\rangle$.

\begin{proposition}\label{prop:cartan-decomposition}
With the notation above,
\begin{equation}\label{eq:cartan-decomposition}
\begin{split}
4\tr(M_\nu\Ecal)={}&\langle(\Theta(\beta)-sI)\nu_0,\nu_0\rangle
+3\sum_i\langle[\beta,T_i],T_i\rangle\\
&+3\sum_{i,j}\langle Cv_{ij},v_{ij}\rangle
+2s|H_\af|^2.
\end{split}
\end{equation}
The sum over $i,j$ is ordered. Every term on the right-hand side is nonnegative. If $\tr(M_\nu\Ecal)=0$, then $v_{ij}=0$ for every $i,j$ and $H|_{\Lambda^3\af}=0$.
\end{proposition}

\begin{proof}
The relation $\Ecal^*J+J\Ecal=2sJ$ gives
\begin{equation}\label{eq:cartan-conformal-action}
q((\Theta(\Ecal)\nu)(X,Y),Z)
=2s\tau(X,Y,Z)+(\rho(\Ecal)\tau)(X,Y,Z).
\end{equation}
We pair \eqref{eq:cartan-conformal-action} with $\tau$ and use \eqref{eq:moment-identity}. The operator $\Ecal$ preserves $\af$, $W$ and $\af^*$, so its insertion action preserves the five orthogonal components. They can therefore be evaluated independently.

On $W^3$, the scalar part of $C=\beta+sI$ contributes $-3s$ through insertion, and the conformal term contributes $2s$. Since $\beta$ is $q$-skew, the remaining tensor corresponds to $(\Theta(\beta)-sI)\nu_0$. This gives the first term of \eqref{eq:cartan-decomposition}.

On $\af W^2$, the two-form $\tau(a_i,\cdot,\cdot)$ is represented by the skew-adjoint operator $J_WT_i$. The conformal and insertion terms give $2sJ_WT_i-CJ_WT_i-J_WT_iC=J_W[\beta,T_i]$, since $\beta J_W=-J_W\beta$. Multiplication by $J_W$ preserves the Hilbert--Schmidt inner product. There are three positions for the $\af$ argument, so the full contraction is $3\sum_i\langle[\beta,T_i],T_i\rangle$.

On $\af^2W$, the one-form action is $2sI-C=sI-\beta$. The relation $(sI-\beta)J_W=J_WC$ converts the pairing to $\langle Cv_{ij},v_{ij}\rangle$. There are three positions for the $W$ slot, and the ordered sum already counts both orders of the horizontal slots. This gives the third term of \eqref{eq:cartan-decomposition}.

On $\af^3$, all insertion terms vanish because $\Ecal|_\af=0$. The remaining contribution is $2s|H_\af|^2$.

On $\af^2\af^*$, the insertion weights are $0+0+2s$. Thus the insertion contributes $-2s$, which cancels the $2s$ conformal term in \eqref{eq:cartan-conformal-action}. This component contributes zero. These five calculations prove \eqref{eq:cartan-decomposition}.

If $\nu_0\ne0$, the first two terms are nonnegative by Lemma~\ref{lem:quadratic-git-inequalities}, and $C>0$ by Proposition~\ref{prop:quadratic-positive-operator}. If $\nu_0=0$, the first two terms vanish and the last two are nonnegative by \eqref{eq:zero-bracket-choice}. Strict positivity of $C$ and $s$ gives the asserted equality consequences.
\end{proof}

When $\nu_0=0$, the identities needed for Proposition~\ref{prop:global-square} reduce to
\begin{equation}\label{eq:zero-quotient-traces}
\|\Ecal\|^2=2m+4r,\qquad
\langle\Fcal,\Ecal\rangle=c(2m+4r),\qquad
\tr(\beta L_0)=0.
\end{equation}
Thus \eqref{eq:extended-operator-pairing} holds with $s=1$, although the nonzero-stratum formula $\tr(\beta L_0)=c$ does not. The Cartan identity \eqref{eq:cartan-decomposition} becomes
\begin{equation}\label{eq:zero-quotient-cartan}
4\tr(M_\nu\Ecal)=3\sum_{i,j}\|v_{ij}\|^2+2|H_{\af}|^2,
\qquad H_\af=H|_{\Lambda^3\af}.
\end{equation}
Both terms in \eqref{eq:zero-quotient-cartan} are nonnegative. These identities include $W=0$ and do not require the quotient $\g/\n$ to be abelian, since its Cartan component has zero contribution. This verifies the complete equality argument in the zero-quotient case directly.

\begin{proof}[Proof of Theorem~\ref{thm:harmonicity}]
Lemma~\ref{lem:invariant-scalar} proves the steady assertion. For $\lambda<0$, Propositions~\ref{prop:global-square} and \ref{prop:cartan-decomposition} imply
\begin{equation}\label{eq:global-equality}
\Fcal=\frac cs\Ecal,\qquad v_{ij}=0,\qquad H|_{\Lambda^3\af}=0.
\end{equation}
The two summands in \eqref{eq:horizontal-bracket-vectors} lie in the independent subspaces $\n$ and $\n^*$. Thus \eqref{eq:global-equality} gives
\begin{equation}\label{eq:general-horizontal-vanishing}
\pr_\n[\af,\af]=0,\qquad H(\af,\af,\g)=0.
\end{equation}
In the original splitting $\g\oplus\g^*$, we have
\begin{equation}\label{eq:corrected-derivation-blocks}
\mathcal K=\begin{pmatrix}
D+\ad_U&0\\
-K_\eta+K_{\iota_UH}&2cI-D-(\ad_U)^*
\end{pmatrix}.
\end{equation}
For $A\in\af$, equations \eqref{eq:extended-operator} and \eqref{eq:global-equality} give $\Fcal A=0$. The lower-left block of $\Fcal=\Sym\mathcal K$ in \eqref{eq:corrected-derivation-blocks} is $(-K_\eta+K_{\iota_UH})/2$. Its value on $A$ vanishes, which yields
\begin{equation}\label{eq:horizontal-defect-vanishing}
\eta(A,Y)=H(U,A,Y)=0\qquad(A\in\af,\ Y\in\g).
\end{equation}
The second equality uses $U\in\af$ and \eqref{eq:general-horizontal-vanishing}. Hence $\eta\in\Lambda^2\n^*$ and $K_\eta$ vanishes on $\af$. Since $B$ is self-adjoint and $B\n=0$, its image lies in $\af$. Thus $K_\eta B=0$ and $K_\eta^*K_\eta B=0$. The harmonicity criterion \cite[Theorem~3.3]{ChChZh2026b}, recalled in \eqref{eq:killing-criterion}, proves $\eta=0$.

Proposition~\ref{prop:nonunimodular-expanding} now gives harmonicity for every non-unimodular soliton. If $\g$ is solvable and nonnilpotent, Lemma~\ref{lem:horizontal} gives $\lambda\le0$, so the assertion just proved applies. Harmonicity in the nilpotent case is \cite[Theorem~1.1]{ChChZh2026a}. It also follows here from $B=0$ and \eqref{eq:killing-criterion}, for every $\lambda$. This includes the abelian and zero-dimensional cases.
\end{proof}

For $H=0$, the reductive complement in the next corollary is established in \cite[Proposition~4.1 and Theorem~4.6(i)]{LaLa2014}. Its solvable specialization, $[\af,\af]=0$, is the standardness assertion of \cite[Theorem~4.8(ii)]{La2011}. We obtain both conclusions below together with the torsion vanishing condition.

\begin{corollary}\label{cor:expanding-reductive-complement}
For an expanding algebraic generalized Ricci soliton on any real metric Lie algebra, the orthogonal complement $\af=\n^\perp$ is a reductive subalgebra and $H(\af,\af,\g)=0$. If $\g$ is solvable, then
\begin{equation}\label{eq:standard}
[\af,\af]=0,\qquad H(\af,\af,\g)=0.
\end{equation}
\end{corollary}

\begin{proof}
Equation~\eqref{eq:general-horizontal-vanishing} shows that $\af$ is a subalgebra. It maps isomorphically onto the reductive quotient $\g/\n$ in Lemma~\ref{lem:general-nilradical}. For a solvable algebra this quotient is abelian, giving \eqref{eq:standard}.
\end{proof}

The positive coefficient in \eqref{eq:global-square} uses $\lambda<0$. Its sign reverses for shrinking solitons, so the equality argument gives no general shrinking conclusion. Proposition~\ref{prop:bi-invariant-harmonicity} proves harmonicity for every soliton constant under the compact semisimple and bi-invariant hypotheses.

\section{Solvable structure and normal extensions}\label{sec:structure}\label{sec:extensions}
In this section, we prove the solvable structure theorem and its converse, then apply them to torsion classes and explicit extensions.

In the remaining sections, $\g$ is solvable unless a statement specifies otherwise. We prove the structural assertion of Theorem~\ref{thm:structure} by using the adjoint representation of the twisted cotangent Lie algebra. We apply the commutator argument of \cite[Lemma~4.7]{La2011} to this algebra. The following identity retains $\eta=\delta H$ and therefore does not require harmonicity.

\begin{lemma}\label{lem:courant-normality}
Suppose that \eqref{eq:soliton} holds with $\lambda=-c<0$ and that
\begin{equation}\label{eq:normality-hypotheses}
 [\af,\af]=0,\qquad H(\af,\af,\g)=0.
\end{equation}
For $A\in\af$, set $T=\ad_A^\nu$. Then
\begin{equation}\label{eq:courant-normality-square}
 0=\frac16\|\Theta(T^*)\nu\|^2
   +2|\iota_A\eta|^2.
\end{equation}
Consequently, $T^*\in\Der(\h)$ and $[T,T^*]=0$.
\end{lemma}

\begin{proof}
Fix $A\in\af$ and write
\begin{equation}\label{eq:normality-operators}
 Q=[T,T^*],\qquad \alpha=\iota_AH,\qquad
 \zeta=\iota_A\eta.
\end{equation}
The invariance of $q$ implies that $T$ and $T^*$ are $q$-skew. Hence $Q$ is both symmetric and $q$-skew. Since $T$ is a derivation, $\Theta(T)\nu=0$. Applying \eqref{eq:moment-identity} to $Q$ gives
\begin{equation}\label{eq:normality-moment}
 4\tr(M_\nu Q)
 =\langle[\Theta(T),\Theta(T^*)]\nu,\nu\rangle
 =\|\Theta(T^*)\nu\|^2.
\end{equation}
Here we have used $\Theta(T)^*=\Theta(T^*)$ for the full bracket inner product. The curvature identity \eqref{eq:curvature} therefore yields
\begin{equation}\label{eq:normality-curvature-trace}
 \tr(PQ)=\frac16\|\Theta(T^*)\nu\|^2
 -\frac12\tr(B_\nu Q)
 -\tr\bigl(\Sym(\ad_U^\nu)Q\bigr).
\end{equation}

We next compute the same trace from the soliton derivation. Lemma~\ref{lem:solvable-preliminaries} gives $DA=0$. Thus \eqref{eq:double-derivation} gives $P=\Sym L-cI$ and $LA=-\zeta$, where $\zeta$ is viewed as an element of $\g^*\subset\h$. Since $L$ is a derivation, $[L,T]=\ad_{-\zeta}^\nu$. Hence
\begin{equation}\label{eq:normality-soliton-trace}
 \tr(PQ)=\tr(LQ)
 =\tr\bigl([L,T]T^*\bigr)
 =\langle d\zeta,\alpha\rangle.
\end{equation}
For the last equality, the only nonzero block of $\ad_{-\zeta}^\nu$ is its lower-left block $K_{d\zeta}$, since $d\zeta(X,Y)=-\zeta([X,Y])$. The lower-left block of $T$ is $K_\alpha$, so the trace is $\tr(K_{d\zeta}K_\alpha^*)=\langle d\zeta,\alpha\rangle$. This is the full contraction pairing of two-forms.

The one-form $A^\flat$ is closed because $A\perp[\g,\g]$. Taking exterior adjoints of
$d(A^\flat\wedge\cdot)+(A^\flat\wedge\cdot)d=0$ gives
$d^\dagger\iota_A+\iota_A d^\dagger=0$. Since contractions anticommute and $\delta=d^\dagger+\iota_U$, we also have $\delta\iota_A+\iota_A\delta=0$. Moreover, $U\in\af$ and \eqref{eq:normality-hypotheses} imply $\iota_U\alpha=0$. It follows that
\begin{equation}\label{eq:normality-codifferential}
 d^\dagger\alpha=\delta\alpha=-\iota_A\eta=-\zeta.
\end{equation}
The full contraction pairing of two-forms is twice the exterior pairing. Using \eqref{eq:norm-convention} and the exterior adjoint $d^\dagger$, equations \eqref{eq:normality-soliton-trace} and \eqref{eq:normality-codifferential} give
\begin{equation}\label{eq:normality-defect-trace}
 \tr(PQ)=2\langle d\zeta,\alpha\rangle_\wedge
 =2\langle\zeta,d^\dagger\alpha\rangle_\wedge=-2|\zeta|^2.
\end{equation}

It remains to eliminate the last two terms in \eqref{eq:normality-curvature-trace}. Condition \eqref{eq:normality-hypotheses} gives $T\af=0$ and
$T\h\subset\n\oplus\n^*$. Therefore $T^*\af=0$ and $Q\af=0$. Since $B_\nu$ vanishes on $W\oplus\af^*$ and maps $\af$ into itself, $\tr(B_\nu Q)=0$. Also
\begin{equation}\label{eq:normality-modular-commutator}
 [\ad_U^\nu,T]=\ad_{[U,A]_\nu}^\nu=0,
\end{equation}
because $U,A\in\af$ and both the bracket and the two-horizontal torsion vanish. By cyclicity of the trace,
$\tr(\Sym(\ad_U^\nu)Q)=\tr([\ad_U^\nu,T]T^*)=0$.
Combining \eqref{eq:normality-curvature-trace} with \eqref{eq:normality-defect-trace} proves \eqref{eq:courant-normality-square}.

All terms in \eqref{eq:courant-normality-square} are nonnegative, so $\Theta(T^*)\nu=0$. Thus $T^*$ is a derivation. Since $T^*A=0$, its derivation identity gives
$[T^*,T]=[T^*,\ad_A^\nu]=\ad_{T^*A}^\nu=0$.
\end{proof}

\begin{proof}[Proof of the structural assertion of Theorem~\ref{thm:structure}]
There is nothing to prove when $\af=0$. Suppose that $\af\ne0$. Equation~\eqref{eq:horizontal} gives $\lambda\leq0$. If $\lambda=0$, it also gives $\Sym(\ad_A)=0$ and $\iota_AH=0$ for every $A\in\af$. Since $\ad_A\g\subset\n$, its adjoint annihilates $\af$. The skew-symmetry of $\ad_A$ then implies $\ad_A\af=0$. Thus $[\af,\af]=0$, and each $\ad_A$ is normal.

Suppose that $\lambda=-c<0$. By \eqref{eq:standard}, the hypotheses of Lemma~\ref{lem:courant-normality} hold. Fix $A\in\af$ and write $V=\ad_A$ and $\alpha=\iota_AH$. In the original splitting $\h=\g\oplus\g^*$,
\begin{equation}\label{eq:normality-blocks}
 T=\begin{pmatrix}V&0\\ K_\alpha&-V^*\end{pmatrix},
 \qquad
 T^*=\begin{pmatrix}V^*&K_\alpha^*\\0&-V\end{pmatrix}.
\end{equation}
The upper-left block of $[T,T^*]=0$ is
\begin{equation}\label{eq:normality-positive-block}
 [V,V^*]-K_\alpha^*K_\alpha=0.
\end{equation}
Taking the trace gives $\|K_\alpha\|^2=0$. Hence $\iota_AH=0$ and $[\ad_A,\ad_A^*]=0$. Since $A$ was arbitrary, $H\in\Lambda^3\n^*$ and all the horizontal adjoint maps are normal. Finally, the maps $\ad_A|_\n$ commute because $\af$ is abelian. This proves the structural assertion.
\end{proof}

The metric on the complement is constrained by the symmetric parts of these commuting normal derivations. For $H=0$ and $\lambda<0$, the following formula is \cite[Theorem~4.8(iv)]{La2011}.

\begin{corollary}\label{cor:horizontal-metric}
Under the hypotheses of Theorem~\ref{thm:structure}, put $V_A=\ad_A|_\n$ for $A\in\af$. Then
\begin{equation}\label{eq:horizontal-metric-condition}
 -\lambda\langle A,B\rangle
 =\tr\bigl(\Sym(V_A)\Sym(V_B)\bigr)
 \qquad(A,B\in\af).
\end{equation}
In particular, $\lambda\leq0$ if $\af\ne0$. In the steady case every $V_A$ is skew-symmetric.
\end{corollary}

\begin{proof}
By the structural assertion of Theorem~\ref{thm:structure}, $\iota_AH=0$ and $\ad_A$ vanishes on $\af$. Equation~\eqref{eq:horizontal} reduces to
$-\lambda|A|^2=\|\Sym(V_A)\|^2$. Polarization gives \eqref{eq:horizontal-metric-condition}.
\end{proof}

We next determine the sign of the soliton constant and give the precise reconstruction statement. This extends the classical construction and characterization in \cite[Proposition~4.3 and Theorem~4.8]{La2011}. The main point is that normal derivations preserving the torsion automatically commute with the nilsoliton derivation. This also ensures that the mean-curvature correction does not change the torsion weight.

\begin{theorem}\label{thm:extension}
Let $(\n,\langle\cdot,\cdot\rangle_\n,H)$ be a generalized nilsoliton with $\Ric_\n-H^2/4=\lambda I+D_\n$ and $\lambda<0$. Let $\af\subset\Der(\n)$ be an abelian subalgebra consisting of normal endomorphisms, and write $S(A)=\Sym A$. Extend $H$ by zero to $\g=\af\ltimes\n$. There is a positive-definite inner product on $\af$ for which the orthogonal extension satisfies \eqref{eq:soliton} with constant $\lambda$ if and only if
\begin{equation}\label{eq:extension-conditions}
\rho(A)H=0\quad(A\in\af),\qquad S|_\af\text{ is injective}.
\end{equation}
This inner product is given by
\begin{equation}\label{eq:extension-metric}
\langle A,B\rangle_\af=-\lambda^{-1}\tr\bigl(S(A)S(B)\bigr).
\end{equation}
If $U\in\af$ is defined by $\langle U,A\rangle_\af=\tr A$, the soliton derivation is
\begin{equation}\label{eq:extension-derivation}
D_\g=0|_\af\oplus\bigl(D_\n-S(U)\bigr)|_\n.
\end{equation}
The nilradical of $\g$ is $\n$. Conversely, every expanding algebraic generalized Ricci soliton on a solvable Lie algebra is obtained by this construction from its induced metric and three-form on the nilradical.
\end{theorem}

The transpose of a normal derivation is again a derivation by \cite[Remark~4.10]{LaLa2014}. For $H=0$, its commutation with $\Ric_\n$ follows from \cite[Remark~2.5]{LaLa2014}. The next lemma includes the three-form in the same argument.

\begin{lemma}\label{lem:normal-stabilizer}
Let $(\n,\mu,\langle\cdot,\cdot\rangle)$ be a nilpotent metric Lie algebra, let $H\in\Lambda^3\n^*$, and put $R_\n=\Ric_\n-H^2/4$. If $A\in\Der(\n)$ is normal and $\rho(A)H=0$, then
\begin{equation}\label{eq:normal-stabilizer}
A^*\in\Der(\n),\qquad \rho(A^*)H=0,\qquad [R_\n,A]=0.
\end{equation}
In particular, both $S(A)=\Sym A$ and $K(A)=\Skew A$ are derivations preserving $H$. If $R_\n=\lambda I+D_\n$, then $[D_\n,A]=0$.
\end{lemma}

\begin{proof}
Let $\pi=\Theta\oplus\rho$ be the representation on the direct sum of bracket tensors and three-forms, with their full contraction inner products, and put $v=(\mu,H/\sqrt3)$. Nilpotency gives $\Ric_\n=M_\mu$. The moment pairing in \cite[Proposition~4.13, equation~(31)]{FuLaSt2024}, restricted to the diagonal action, gives
\begin{equation}\label{eq:joint-moment}
4\tr(R_\n B)=\langle\pi(B)v,v\rangle
\qquad(B\in\operatorname{End}(\n)).
\end{equation}
Indeed, $\langle\rho(B)H,H\rangle=-3\tr(H^2B)$, so $\langle\pi(B)v,v\rangle=4\tr(M_\mu B)-\tr(H^2B)$. This verifies the normalization in \eqref{eq:joint-moment}. The representation satisfies $\pi(B)^*=\pi(B^*)$. Since $[A,A^*]=0$, the operator $\pi(A)$ is normal. The hypotheses give $\pi(A)v=0$, so $\pi(A^*)v=0$ as well. This proves the first two assertions of \eqref{eq:normal-stabilizer}.

For arbitrary $B\in\operatorname{End}(\n)$, equation \eqref{eq:joint-moment} now yields
\begin{equation}\label{eq:normal-stabilizer-commutator}
\begin{aligned}
4\tr([R_\n,A]B)
&=\langle\pi([A,B])v,v\rangle\\
&=\langle\pi(B)v,\pi(A^*)v\rangle
-\langle\pi(B)\pi(A)v,v\rangle=0.
\end{aligned}
\end{equation}
The trace pairing on endomorphisms is nondegenerate, so \eqref{eq:normal-stabilizer-commutator} implies $[R_\n,A]=0$. Taking linear combinations of $A$ and $A^*$ gives the assertions about $S(A)$ and $K(A)$. The last assertion follows from $D_\n=R_\n-\lambda I$.
\end{proof}

We record the curvature and differential formulas used in both directions. Let $\af\subset\Der(\n)$ be abelian and consist of normal operators preserving $H$. Give $\af$ any positive-definite inner product and form the orthogonal semidirect product $\g=\af\ltimes\n$. Thus $[A,X]=AX$, $[\af,\af]=0$, and the bracket on $\n$ is unchanged. Write $A_1,\ldots,A_r$ for an orthonormal basis of $\af$ and $a^1,\ldots,a^r$ for the corresponding dual forms. The mean-curvature vector is the element $U\in\af$ defined by $\langle U,A\rangle=\tr A$. For a form $\omega$ on $\n$, extended by zero on horizontal slots, the Chevalley--Eilenberg differential and the geometric codifferential satisfy
\begin{equation}\label{eq:extension-differentials}
d_\g\omega=d_\n\omega+\sum_i a^i\wedge\rho(A_i)\omega,
\qquad \delta_\g\omega=d_\n^\dagger\omega.
\end{equation}
The first identity follows by evaluating on vertical vectors and on one horizontal vector. For the second, the differential does not decrease the number of horizontal slots, so its exterior adjoint on a vertical form is $d_\n^\dagger$. The mean-curvature term vanishes because $\iota_U\omega=0$.

The Ricci formulas in \cite[Section~4, equation~(25)]{La2011} reduce to
\begin{equation}\label{eq:extension-ricci}
\begin{aligned}
\langle\Ric_\g A,B\rangle&=-\tr\bigl(S(A)S(B)\bigr),\qquad
\langle\Ric_\g A,X\rangle=0,\\
\Ric_\g|_\n&=\Ric_\n-S(U),
\qquad A,B\in\af,\quad X\in\n.
\end{aligned}
\end{equation}
To verify these formulas, substitute the semidirect-product bracket into \eqref{eq:moment-operator} and \eqref{eq:ricci-decomposition}. The vertical block before imposing normality is $\Ric_\n+\frac12\sum_i[A_i,A_i^*]-S(U)$, and the commutators vanish. The mixed block is $-\frac12\tr(A^*\ad_X|_\n)$. By Lemma~\ref{lem:normal-stabilizer}, $A^*$ is a derivation. Lie's theorem simultaneously triangularizes the action of the solvable Lie algebra $\R A^*\ltimes\n$ on $\n_\mathbb C$. The operator $\ad_X|_\n$ is nilpotent, so it is strictly upper triangular in this basis and the mixed trace vanishes. The horizontal block is $-\frac12\tr(A^*B)-\frac12\tr(AB)$, which is the first expression in \eqref{eq:extension-ricci}.

In this semidirect product, $\n$ is the full nilradical. Indeed, Lie's theorem simultaneously triangularizes the adjoint action of $\g_\mathbb C$ on $\n_\mathbb C$. For $A\in\af$ and $X\in\n$, the diagonal of $\ad_{A+X}|_\n$ is the diagonal of $A$, since $\ad_X|_\n$ is nilpotent. If $A\ne0$, normality implies that $A$ has a nonzero eigenvalue. On the other hand, an element of the nilradical acts nilpotently on every invariant subspace of that nilpotent ideal, including $\n$. Thus $A+X$ cannot belong to the nilradical. The reverse inclusion holds because $\n$ is a nilpotent ideal. This argument does not use $\rho(A)H=0$.

We now apply these formulas to the structural decomposition already proved in Theorem~\ref{thm:structure}. Generalized nilsolitons with nonzero bracket or torsion are expanding by \cite[Proposition~7.7]{FuLaSt2024}. For solvable algebras, the nonzero-torsion sign conclusion is \cite[Theorem~4.2]{Mu2026}, and the nonflat torsion-free conclusion follows from \cite[Proposition~4.6]{La2011}. In the steady case, Lemma~\ref{lem:invariant-scalar} gives a Bismut--Ricci-flat metric, whose flatness and vanishing torsion also follow from \cite[Remark~2.17 and Corollary~3.3]{CoFrGa2024}. The following argument derives these conclusions directly from the nilradical and includes the flat cases.

\begin{corollary}\label{cor:expanding-solvable}
Let $(\g,\langle\cdot,\cdot\rangle,H)$ be an algebraic generalized Ricci soliton on a solvable Lie algebra, with constant $\lambda$ and derivation $D$. If $H\ne0$ or the corresponding left-invariant metric is nonflat, then $\lambda<0$. If $\lambda=0$, then $H=0$, the metric is flat and $D=0$. If $\lambda>0$, then $\g$ is abelian, $H=0$ and $D=-\lambda I$.
\end{corollary}

\begin{proof}
Let $\n$ be the nilradical and $\af=\n^\perp$. By Theorem~\ref{thm:harmonicity} and the structural assertion of Theorem~\ref{thm:structure}, $\g=\af\ltimes\n$, $\af$ is abelian, $H\in\Lambda^3\n^*$, and the operators $V_A=\ad_A|_\n$ are commuting and normal. Their action is faithful, since an element of $\af$ acting trivially on $\n$ would be central and hence belong to $\n$. We identify $\af$ with its image in $\Der(\n)$. Equation~\eqref{eq:extension-differentials} and harmonicity give $d_\n H=d_\n^\dagger H=0$ and $\rho(V_A)H=0$ for all $A\in\af$.

Write $S(U)=\Sym(V_U)$ for the mean-curvature vector $U\in\af$. Lemma~\ref{lem:normal-stabilizer} shows that $S(U)$ is a symmetric derivation of $\n$ preserving $H$. Since $D$ is symmetric and preserves $\n$, the endomorphism
\begin{equation}\label{eq:induced-nilsoliton}
D_\n=D|_\n+S(U)
\end{equation}
is a symmetric derivation of $\n$. The vertical block of \eqref{eq:extension-ricci} gives $R_\n:=\Ric_\n-H^2/4=\lambda I+D_\n$. The harmonic torsion equation and $\rho(S(U))H=0$ give $\rho(D_\n)H=2\lambda H$. Thus the induced data on $\n$ form a generalized nilsoliton with the same constant, without any assumption on its sign.

Let $\mu$ be the bracket of $\n$. Nilpotency and \eqref{eq:moment-identity} imply $\tr(\Ric_\n D_\n)=0$ and $\tr\Ric_\n=-\|\mu\|^2/4$. The insertion identity $\langle\rho(D_\n)H,H\rangle=-3\tr(H^2D_\n)$ gives $\tr(H^2D_\n)=-2\lambda|H|^2/3$. Consequently,
\begin{equation}\label{eq:nilsoliton-sign}
\begin{aligned}
\tr R_\n&=-\tfrac14\bigl(\|\mu\|^2+|H|^2\bigr),\\
\|R_\n\|^2
&=\lambda\tr R_\n+\tr(R_\n D_\n)
=-\frac{\lambda}{4}\left(\|\mu\|^2+\frac13|H|^2\right).
\end{aligned}
\end{equation}
If $(\mu,H)\ne(0,0)$, the first identity in \eqref{eq:nilsoliton-sign} shows that $R_\n\ne0$, and the second gives $\lambda<0$.

It remains to consider $\mu=0$ and $H=0$. If $\af=0$, then $\g$ is abelian and the metric equation gives $D=-\lambda I$. Suppose that $\af\ne0$. Equation~\eqref{eq:horizontal-metric-condition} implies $\lambda\le0$. When $\lambda=0$, every $V_A$ is skew-adjoint. Both $\af$ and $\n$ are now abelian, and the Koszul formula gives
\begin{equation}\label{eq:flat-solvable-connection}
\nabla_{A+X}(B+Y)=V_A Y
\qquad(A,B\in\af,\ X,Y\in\n).
\end{equation}
Since $[\g,\g]\subset\n$ and \eqref{eq:flat-solvable-connection} gives $\nabla_X=0$ for $X\in\n$, the curvature is $\operatorname{Rm}(A+X,B+Y)(C+Z)=[V_A,V_B]Z=0$. Thus the metric is flat, and the soliton equation gives $D=0$. These alternatives prove every assertion.
\end{proof}

The flat decomposition in the steady case agrees with Milnor's characterization of flat left-invariant metrics \cite[Theorem~1.5]{Mi1976}. It also gives the following restriction under complete solvability.

\begin{corollary}\label{cor:steady-completely-solvable}
A completely solvable Lie algebra admits a steady algebraic generalized Ricci soliton if and only if it is abelian. Every such solution has $H=D=0$.
\end{corollary}

\begin{proof}
By the proof of Corollary~\ref{cor:expanding-solvable}, a steady solution has $H=D=0$, both $\af$ and $\n$ are abelian, and each $V_A=\ad_A|_\n$ is skew-adjoint. Complete solvability makes its eigenvalues real, so $V_A=0$ for every $A\in\af$. Thus $\g$ is abelian. Conversely, every positive-definite inner product on an abelian Lie algebra gives a steady solution with $H=D=0$.
\end{proof}

\begin{remark}\label{rem:flat-gaussian}
The flat torsion-free exception is necessary. On an abelian metric Lie algebra with $H=0$, the endomorphism $D=-\lambda I$ is a symmetric derivation for every $\lambda\in\R$, and both equations in \eqref{eq:soliton} hold. On the simply connected group, these give the Gaussian expanding, steady and shrinking presentations of the Euclidean metric. In \eqref{eq:self-similar-family}, the automorphisms cancel the scaling, so the resulting flow is stationary. For a nonabelian flat metric Lie algebra with $H=0$, the metric equation still gives $D=-\lambda I$, but the derivation identity forces $\lambda[X,Y]=0$ for all $X,Y\in\g$. Hence only $\lambda=0$ is possible for that Lie algebra presentation.
\end{remark}

\begin{proof}[Proof of Theorem~\ref{thm:extension}]
Suppose first that \eqref{eq:extension-conditions} holds. Since $\lambda<0$ and $S|_\af$ is injective, equation \eqref{eq:extension-metric} defines a positive-definite inner product. Lemma~\ref{lem:normal-stabilizer} gives $[D_\n,A]=0$ for every $A\in\af$, and $S(U)$ is a symmetric derivation preserving $H$. A commuting family of normal operators is simultaneously unitarily diagonalizable after complexification. Consequently, $[A,B^*]=0$ for $A,B\in\af$, and hence $[S(U),A]=0$. Therefore \eqref{eq:extension-derivation} defines a symmetric derivation of $\g$.

Equations \eqref{eq:extension-metric} and \eqref{eq:extension-ricci} give $\Ric_\g-H^2/4=\lambda I+D_\g$. The nilsoliton has $d_\n H=d_\n^\dagger H=0$ by \cite[Theorem~1.1]{ChChZh2026a}. Equation \eqref{eq:extension-differentials} and the condition $\rho(A)H=0$ therefore give $d_\g H=\delta_\g H=0$. Finally,
\begin{equation}\label{eq:extension-weight}
\rho(D_\g)H
=\rho(D_\n)H-\rho(S(U))H
=2\lambda H.
\end{equation}
This proves both soliton equations.

Conversely, suppose that the orthogonal semidirect product in the statement satisfies \eqref{eq:soliton} with constant $\lambda$. Closedness gives $\rho(A)H=0$ by \eqref{eq:extension-differentials}. The nilradical is $\n$ by the preceding argument. Corollary~\ref{cor:horizontal-metric} therefore gives \eqref{eq:extension-metric}, which forces $S|_\af$ to be injective. The metric equation and \eqref{eq:extension-ricci} then determine $D_\g$ as in \eqref{eq:extension-derivation}.

For the final assertion, start with any expanding algebraic generalized Ricci soliton on a solvable Lie algebra. Theorem~\ref{thm:harmonicity} and the structural assertion of Theorem~\ref{thm:structure} give an orthogonal decomposition $\g=\af\ltimes\n$ with $\af$ abelian, $H\in\Lambda^3\n^*$, and commuting normal actions $A\mapsto\ad_A|_\n$. This action is faithful. Otherwise a nonzero element of $\af$ would be central and hence belong to $\n$. We identify $\af$ with its image in $\Der(\n)$.

The proof of Corollary~\ref{cor:expanding-solvable} shows that the induced data on $\n$ form a generalized nilsoliton with derivation \eqref{eq:induced-nilsoliton}, and that $\rho(A)H=0$ for every $A\in\af$. Equations~\eqref{eq:extension-metric} and~\eqref{eq:extension-derivation} recover the original metric and derivation. This also proves the reconstruction assertion of Theorem~\ref{thm:structure}.
\end{proof}

\begin{corollary}\label{cor:extension-existence}
Fix generalized nilsoliton data $(\n,\langle\cdot,\cdot\rangle,H,\lambda)$ with $\lambda<0$, and set
\begin{equation}\label{eq:symmetric-torsion-stabilizer}
\mathfrak p_H=\{S\in\Der(\n):S=S^*,\ \rho(S)H=0\}.
\end{equation}
For $r\ge1$, an extension in Theorem~\ref{thm:extension} with $\dim\af=r$ exists if and only if $\mathfrak p_H$ contains an $r$-dimensional subspace of pairwise commuting operators. In particular, a nonzero complement exists if and only if $\mathfrak p_H\ne\{0\}$.
\end{corollary}

\begin{proof}
For any admissible $\af$, Lemma~\ref{lem:normal-stabilizer} gives $S(\af)\subset\mathfrak p_H$. The symmetric parts commute because the original operators commute and are normal. Injectivity of $S|_\af$ gives $\dim S(\af)=\dim\af$. Conversely, any commuting subspace of $\mathfrak p_H$ satisfies \eqref{eq:extension-conditions}, so Theorem~\ref{thm:extension} applies.
\end{proof}

For fixed nilsoliton data, the admissible actions in Theorem~\ref{thm:extension} can be compared through the orthogonal stabilizer. For $H=0$ and symmetric actions, this equivalence is \cite[Proposition~5.3]{La2011}. Let
\begin{equation}\label{eq:extension-compact-group}
K_H=\{k\in\Aut(\n)\cap\mathrm O(\n):k^*H=H\}.
\end{equation}
For fixed $(\n,\langle\cdot,\cdot\rangle,H,\lambda)$, two admissible subalgebras $\af_1,\af_2$ give orthogonally isomorphic Lie algebras with three-form if and only if $\af_2=k\af_1k^{-1}$ for some $k\in K_H$. Indeed, such a conjugation preserves the trace metric \eqref{eq:extension-metric} and extends to the required isomorphism. Conversely, an orthogonal Lie algebra isomorphism preserving the three-form preserves the nilradical and its orthogonal complement. Its restriction $k$ to $\n$ belongs to $K_H$, and the bracket identity gives the stated conjugation. This classifies the extensions of the fixed nilsoliton data.

The skew-symmetric parts affect the Lie algebra presentation but can be removed from the simply connected metric and three-form. For ordinary solvsolitons, this follows from \cite[equation~(37) and Corollary~4.10]{La2011}. We verify below that the same change also preserves the three-form.

\begin{proposition}\label{prop:extension-isometry}
In Theorem~\ref{thm:extension}, replace the action of each $A\in\af$ by $S(A)$ and retain the metric on $\af$. The simply connected groups associated with these two semidirect products are isometric by a diffeomorphism preserving the extended three-form. In particular, every simply connected expanding algebraic generalized Ricci soliton in the theorem has a completely solvable presentation.
\end{proposition}

\begin{proof}
Let $N$ be the simply connected group with Lie algebra $\n$. Use $\af\times N$ as coordinates for both semidirect products, with multiplication $(a,x)(b,y)=(a+b,x\Phi_a(y))$. The differential of $\Phi_a$ at the identity of $N$ is $e^a$ for the original action and $e^{S(a)}$ for the symmetric action. The symmetric and skew-symmetric parts of the commuting normal family commute with each other. Lemma~\ref{lem:normal-stabilizer} shows that both parts are derivations preserving $H$.

For $v\in\af$ and $X\in T_xN$, left translation identifies the tangent vector $(v,X)$ at $(a,x)$ with $(v,e^{-a}(dL_{x^{-1}})_xX)$ at the identity. Hence its squared norm in the original metric is
\begin{equation}\label{eq:extension-isometry-metric}
|v|^2+\bigl|e^{-a}(dL_{x^{-1}})_xX\bigr|^2
=|v|^2+\bigl|e^{-S(a)}(dL_{x^{-1}})_xX\bigr|^2.
\end{equation}
Here $e^{-K(a)}$ is orthogonal and commutes with $e^{-S(a)}$. The right-hand side is the metric for the symmetric action. Moreover, $e^a$ and $e^{S(a)}$ both preserve $H$, so the extended three-forms also agree in these coordinates. The identity map on $\af\times N$ is therefore the required isometry. The symmetric action has real eigenvalues, and adding an inner nilpotent action does not change the diagonal in a simultaneous triangularization. Thus the second Lie algebra is completely solvable.
\end{proof}

As in \cite[Remark~4.5]{La2011}, Theorem~\ref{thm:extension} permits a flat abelian nilradical with $H=0$, any chosen $\lambda<0$ and $D_\n=-\lambda I$.

The identity $H(\af,\g,\g)=0$ gives an obstruction to prescribed invariant cohomology classes. All cohomology groups below are computed in the invariant Chevalley--Eilenberg complex. For nilpotent Lie algebras, the nontriviality of the class of nonzero soliton torsion appears in \cite[Proposition~6.5]{ChChZh2026b}. The restriction to the nilradical gives the following solvable version.

\begin{proposition}\label{prop:cohomology-obstruction}
Let $\g$ be solvable with nilradical $\n$, and let $\operatorname{res}\colon H^3_{\mathrm{CE}}(\g)\to H^3_{\mathrm{CE}}(\n)$ be the restriction map. If $(g,H)$ is an algebraic generalized Ricci soliton and $\operatorname{res}[H]=0$, then $H=0$. Consequently, no nonzero class in $\ker\operatorname{res}$ is represented by the torsion of such a soliton. If $H^3_{\mathrm{CE}}(\n)^{\g/\n}=0$, every algebraic generalized Ricci soliton on $\g$ is an ordinary solvsoliton.
\end{proposition}

\begin{proof}
Set $\af=\n^\perp$. By Theorems~\ref{thm:harmonicity} and \ref{thm:structure}, the form $H$ is vertical and $\delta_\g H=0$. Equation~\eqref{eq:extension-differentials} gives $d_\n H=d_\n^\dagger H=0$. If $\operatorname{res}[H]=0$, write $H|_\n=d_\n B$. Taking the exterior inner product gives
\begin{equation}\label{eq:exact-torsion-vanishing}
\langle H,H\rangle_{\Lambda^3\n^*}
=\langle d_\n B,H\rangle_{\Lambda^3\n^*}
=\langle B,d_\n^\dagger H\rangle_{\Lambda^2\n^*}=0.
\end{equation}
Thus $H=0$ on $\g$, proving the first two assertions.

The adjoint action of $\g$ on $\n$ induces an action on $H^3_{\mathrm{CE}}(\n)$. For $X\in\n$, Cartan's formula gives $\rho(\ad_X)=d_\n\iota_X+\iota_Xd_\n$, so inner derivations act trivially on cohomology. The action therefore factors through $\g/\n$. Since $d_\g H=0$ and $H$ is vertical, equation~\eqref{eq:extension-differentials} gives $\rho(\ad_A|_\n)H=0$ for every $A\in\af$. Hence $[H|_\n]\in H^3_{\mathrm{CE}}(\n)^{\g/\n}$. The last assertion follows from the first.
\end{proof}

\begin{example}\label{ex:cohomology-obstruction}
Let $\g=\mathfrak{aff}(\R)\oplus\R^2$ have basis $(A,X,Z_1,Z_2)$ and the single nonzero bracket $[A,X]=X$. In the dual basis $(a,x,z^1,z^2)$, the only nonzero differential is $dx=-a\wedge x$. Thus
\begin{equation}\label{eq:affine-cohomology}
H^3_{\mathrm{CE}}(\g)=\R[a\wedge z^1\wedge z^2],
\qquad \n=\operatorname{span}\{X,Z_1,Z_2\}.
\end{equation}
Indeed, $a\wedge x\wedge z^i=-d(x\wedge z^i)$, while $x\wedge z^1\wedge z^2$ is not closed. The remaining closed basis three-form is not exact. Its restriction to $\n$ vanishes, so Proposition~\ref{prop:cohomology-obstruction} excludes every nonzero class in \eqref{eq:affine-cohomology}. In particular, every algebraic generalized Ricci soliton on this Lie algebra has zero torsion. This obstruction concerns invariant cohomology. The simply connected group is diffeomorphic to $\R^4$ and has trivial third de Rham cohomology.
\end{example}

The next torsion-free family is classical and appears, up to scaling and a change of basis, in \cite[Section~6.1]{La2011}. We record its curvature and Killing operator for use in the subsequent example.

\begin{example}\label{ex:rotating-einstein}
Fix $a\geq0$ and $b\ne0$. Let $\mathfrak s_{a,b}$ have an orthonormal basis $e_0,e_1,e_2$ and nonzero brackets
\begin{equation}\label{eq:rotating-einstein-brackets}
[e_0,e_1]=ae_1+be_2,
\qquad
[e_0,e_2]=-be_1+ae_2.
\end{equation}
The ideal $\mathfrak v=\operatorname{span}\{e_1,e_2\}$ is abelian, and the action of $e_0$ on $\mathfrak v$ is
$A=\left(\begin{smallmatrix}a&-b\\ b&a\end{smallmatrix}\right)$.
Thus $A$ is normal, $\operatorname{Sym}A=aI$, and $[A,A^*]=0$. Formula \eqref{eq:extension-ricci} gives
\begin{equation}\label{eq:rotating-einstein-curvature}
\operatorname{Ric}(e_0)=-\operatorname{tr}\bigl((\operatorname{Sym}A)^2\bigr)e_0=-2a^2e_0,
\qquad
\operatorname{Ric}|_{\mathfrak v}
=\tfrac12[A,A^*]-(\operatorname{tr}A)\operatorname{Sym}A
=-2a^2I.
\end{equation}
The mixed Ricci components vanish. Therefore $H=0$, $\lambda=-2a^2$, and $D=0$ satisfy both algebraic soliton equations. All closedness and coclosedness conditions are automatic for $H=0$.

Since the eigenvalues of $A$ are $a\pm ib$, the Lie algebra is not completely solvable. If $a>0$, this is an expanding Einstein solvsoliton. If $a=0$, the metric is flat and the soliton is steady. In particular, the latter case gives the standard flat metric on the simply connected group with Lie algebra $\mathfrak e(2)$. The Killing operator is
$B=\operatorname{diag}(2(a^2-b^2),0,0)$, so it has a negative direction whenever $|b|>a$.
\end{example}

\begin{example}\label{ex:rotating-nonzero-torsion}
Fix $a>0$ and $b\ne0$. Take the orthogonal direct product
$\mathfrak g=\mathfrak s_{a,b}\oplus\mathfrak z$, where $\mathfrak z$ is abelian with orthonormal basis $z_1,z_2,z_3$. Let $z^1,z^2,z^3$ denote the corresponding dual forms, extended by zero on $\mathfrak s_{a,b}$, and set
\begin{equation}\label{eq:rotating-torsion-data}
H=\frac{2a}{\sqrt3}\,z^1\wedge z^2\wedge z^3,
\qquad
\lambda=-2a^2,
\qquad
D=0|_{\mathfrak s_{a,b}}\oplus\frac{4a^2}{3}I|_{\mathfrak z}.
\end{equation}
We verify the complete soliton system. The forms $z^j$ are closed, so $dH=0$. Since $\mathfrak z$ is central, $(d\beta)(z_1,z_2,z_3)=0$ for every $\beta\in\Lambda^2\mathfrak g^*$. Therefore $H$ is orthogonal to $d(\Lambda^2\mathfrak g^*)$. Hence $d^\dagger H=0$. The mean-curvature vector is $U=2ae_0$, so $\iota_UH=0$ as well. It follows that $\delta H=d^\dagger H+\iota_UH=0$.

By the product structure and \eqref{eq:rotating-einstein-curvature},
\begin{equation}\label{eq:rotating-torsion-tensors}
\operatorname{Ric}=(-2a^2I)|_{\mathfrak s_{a,b}}\oplus0|_{\mathfrak z},
\qquad
H^2=0|_{\mathfrak s_{a,b}}\oplus\frac{8a^2}{3}I|_{\mathfrak z}.
\end{equation}
The factor $2$ in the contraction of a unit volume three-form is included in the second identity. Subtracting $H^2/4$ gives
\begin{equation}\label{eq:rotating-torsion-metric-equation}
\operatorname{Ric}-\tfrac14H^2
=(-2a^2I)|_{\mathfrak s_{a,b}}\oplus\left(-\frac{2a^2}{3}I\right)|_{\mathfrak z}
=\lambda I+D.
\end{equation}
The endomorphism $D$ is a symmetric derivation because it vanishes on the nonabelian direct factor and is scalar on the central abelian factor. Each slot of $H$ has $D$-weight $4a^2/3$. Consequently,
\begin{equation}\label{eq:rotating-torsion-form-equation}
\rho(D)H=-4a^2H=2\lambda H,
\qquad
 d\delta H=0=2\lambda H-\rho(D)H.
\end{equation}
Thus \eqref{eq:rotating-torsion-data} defines an expanding algebraic generalized Ricci soliton with nonzero harmonic torsion.

Its nilradical is the abelian ideal
$\mathfrak n=\operatorname{span}\{e_1,e_2,z_1,z_2,z_3\}$.
The orthogonal complement is $\mathbb Re_0$, whose action on $\mathfrak n$ is $A\oplus0$ and is normal. The three-form $H$ vanishes whenever one argument lies in $\mathfrak z^\perp$, where $\mathfrak z\subset\mathfrak n$ is central. The Lie algebra is not completely solvable, and its Killing operator again has a negative direction if $|b|>a$. These examples satisfy the original equations directly, including the geometric codifferential in the non-unimodular case.
\end{example}

The following four-dimensional models include those of \cite[Propositions~6.2 and~6.4]{Mu2026}. They illustrate the normal extension criterion and the difference between normal and symmetric actions.

\begin{example}\label{ex:volume-extension}
Let $\n=\operatorname{span}\{x,y,z\}$ be abelian with the displayed basis orthonormal. Fix $p>0$ and $b\in\R$, and put $H=2p\,x^*\wedge y^*\wedge z^*$. Then $H^2=8p^2I_\n$, so the nilsoliton equations hold with $\lambda=-6p^2$ and $D_\n=4p^2I_\n$. Define a derivation $A$ by
\begin{equation}\label{eq:volume-extension-action}
Ax=px+by,\qquad Ay=-bx+py,\qquad Az=-2pz.
\end{equation}
The operator $A$ is normal and has trace zero. Hence it preserves the volume form $H$, and $\tr(S(A)^2)=6p^2$. The trace metric \eqref{eq:extension-metric} makes $A$ a unit vector in $\af=\R A$, and $U=0$. Theorem~\ref{thm:extension} gives a four-dimensional soliton with
\begin{equation}\label{eq:volume-extension-soliton}
\Ric_\g=(-6p^2I)|_\af\oplus0|_\n,
\qquad D_\g=0|_\af\oplus4p^2I|_\n.
\end{equation}
The torsion is harmonic and $\rho(D_\g)H=-12p^2H=2\lambda H$. If $b\ne0$, the adjoint spectrum is not real. The Killing operator has the single possibly nonzero entry $B(A,A)=6p^2-2b^2$, which is negative for $b^2>3p^2$. The center of the extended Lie algebra is zero because $A$ is invertible on $\n$. Thus nonzero torsion and a negative Killing direction can also occur on a Lie algebra with zero center.
\end{example}

\section{Rank and moduli of symmetric extensions}\label{sec:rank-moduli}
In this section, we derive a sharp rank bound and classify symmetric extensions of fixed generalized nilsoliton data.

\begin{theorem}\label{thm:sharp-rank}
Let $(\g,g,H)$ be a solvable algebraic generalized Ricci soliton with $H\ne0$, let $\n$ be its nilradical, and put $m=\dim\n$ and $\af=\n^\perp$. Then
\begin{equation}\label{eq:sharp-rank-bound}
\dim\af\le m-1,\qquad \dim\n\ge\frac{\dim\g+1}{2}.
\end{equation}
Equality holds if and only if, in orthonormal coordinates on $\n$ and after identifying $\af$ with its adjoint action on $\n$, the data are
\begin{equation}\label{eq:sharp-rank-model}
\begin{gathered}
\n=\R^m\text{ abelian},\qquad H=h\,e^1\wedge e^2\wedge e^3,\qquad h>0,\\
\af=\mathfrak d_0:=\{\operatorname{diag}(a_1,\ldots,a_m):a_1+a_2+a_3=0\},\\
\lambda=-\tfrac32h^2,\qquad
\langle A,B\rangle_\af=\frac{2}{3h^2}\tr(AB).
\end{gathered}
\end{equation}
For every $m\ge3$, the model \eqref{eq:sharp-rank-model} is a soliton of dimension $2m-1$. In particular, every equality case has symmetric horizontal action and is completely solvable.
\end{theorem}

We now prove the dimension bound and determine all equality cases.

\begin{proof}[Proof of Theorem~\ref{thm:sharp-rank}]
Since $H\ne0$, Corollary~\ref{cor:expanding-solvable} gives $\lambda<0$. Use Theorem~\ref{thm:extension} to identify $\af$ with its faithful normal action on $\n$. In particular, $H\in\Lambda^3\n^*$ and $m\ge3$. By Lemma~\ref{lem:normal-stabilizer}, $\mathfrak b:=S(\af)$ consists of commuting symmetric derivations preserving $H$, and $\dim\mathfrak b=\dim\af$. Choose an orthonormal basis $e_1,\ldots,e_m$ simultaneously diagonalizing $\mathfrak b$. If $H(e_i,e_j,e_k)\ne0$ with $i<j<k$, then $\rho(B)H=0$ for $B=\operatorname{diag}(b_1,\ldots,b_m)\in\mathfrak b$ gives $b_i+b_j+b_k=0$. Thus $\mathfrak b$ lies in a proper hyperplane of the diagonal algebra, so $\dim\af\le m-1$. Since $\dim\g=m+\dim\af$, the two inequalities in \eqref{eq:sharp-rank-bound} are equivalent.

Suppose that $\dim\af=m-1$. Relabel the basis so that $H(e_1,e_2,e_3)\ne0$. Then $\mathfrak b=\mathfrak d_0$. Any other nonzero coefficient $H(e_i,e_j,e_k)$ with $i<j<k$ would force the linear functional $b_i+b_j+b_k$ to vanish on $\mathfrak d_0$. Its coefficient vector must therefore be a multiple of $(1,1,1,0,\ldots,0)$. Since both coefficient vectors have three entries equal to one and all others zero, the triples agree. Changing the sign of one basis vector if necessary gives $H=h\,e^1\wedge e^2\wedge e^3$ with $h>0$.

Write $\mu_{ij}^k=\langle[e_i,e_j],e_k\rangle$. Because every element of $\mathfrak d_0$ is a derivation, a nonzero coefficient $\mu_{ij}^k$ would imply
\begin{equation}\label{eq:rank-equality-bracket-weight}
b_k-b_i-b_j=0\qquad
\text{for all }(b_1,\ldots,b_m)\text{ with }b_1+b_2+b_3=0.
\end{equation}
The coefficient vector on the left has integer entries with sum $-1$. If it were a multiple of $(1,1,1,0,\ldots,0)$, that multiple would be $-1/3$, which is impossible. Equation \eqref{eq:rank-equality-bracket-weight} therefore forces $\mu=0$.

It remains to recover the original normal action, including its skew-symmetric parts. Commuting normal operators commute with one another's adjoints. Hence $K(A)$ commutes with every element of $S(\af)=\mathfrak d_0$ for each $A\in\af$. The matrix $\operatorname{diag}(1,2,-3,4,\ldots,m)$ belongs to $\mathfrak d_0$ and has distinct eigenvalues, with the last entries omitted when $m=3$. Any commuting matrix is diagonal in this basis, so its skew-symmetric centralizer is zero. Therefore $K(A)=0$ for every $A$, and the original action is symmetric and equals $\mathfrak d_0$.

Let $P$ be the orthogonal projection onto $\operatorname{span}\{e_1,e_2,e_3\}$. Since $\mu=0$ and $H=h\,e^1\wedge e^2\wedge e^3$, one has $R_\n=-h^2P/2$. The induced nilsoliton derivation is $D_\n=R_\n-\lambda I$. Its torsion weight equation gives $3(h^2/2+\lambda)=2\lambda$, so
\begin{equation}\label{eq:rank-equality-nilsoliton}
\lambda=-\tfrac32h^2,\qquad
D_\n=h^2P+\tfrac32h^2(I-P).
\end{equation}
The horizontal metric in \eqref{eq:sharp-rank-model} follows from \eqref{eq:extension-metric}.

Conversely, for any $m\ge3$ and $h>0$, the abelian nilpotent data with $H=h\,e^1\wedge e^2\wedge e^3$ satisfy the nilsoliton equations with \eqref{eq:rank-equality-nilsoliton}. Every element of $\mathfrak d_0$ is symmetric and preserves $H$. Theorem~\ref{thm:extension} therefore gives precisely \eqref{eq:sharp-rank-model}, with full nilradical $\R^m$. For $A\in\mathfrak d_0$, we have $\tr(PA)=0$, so the trace metric gives $U=\tfrac32h^2(I-P)\in\mathfrak d_0$. The full soliton derivation is therefore zero on $\af$ and equals $h^2P$ on $\n$. This also verifies the model directly, including $m=3$, when $U=0$.
\end{proof}

For classical solvsolitons, Will gives the Grassmannian quotient and identifies its fibres through a finite Weyl group in \cite[Section~3, equation~(5)]{Wi2011}. That construction applies in every dimension. We extend it to the simultaneous stabilizer of the nilpotent bracket and the three-form. We use $\mathfrak p_H$ and $K_H$ from \eqref{eq:symmetric-torsion-stabilizer} and \eqref{eq:extension-compact-group}. The nilsoliton data and $\lambda$ are fixed, and the equivalence is orthogonal Lie algebra isomorphism preserving $H$.

\begin{theorem}\label{thm:all-rank-moduli}
Fix generalized nilsoliton data $(\n,\langle\cdot,\cdot\rangle,H,\lambda)$ with $\lambda<0$. Let $\mathfrak t$ be a maximal abelian subspace of $\mathfrak p_H$, and put $r=\dim\mathfrak t$. Define
\begin{equation}\label{eq:extension-weyl-group}
W_H=N_{K_H}(\mathfrak t)/Z_{K_H}(\mathfrak t),
\end{equation}
where $N_{K_H}(\mathfrak t)$ preserves $\mathfrak t$ under conjugation and $Z_{K_H}(\mathfrak t)$ fixes it pointwise. Then $W_H$ is finite. For each $0\le k\le r$, the $k$-dimensional symmetric extensions in Theorem~\ref{thm:extension}, modulo orthogonal Lie algebra isomorphisms preserving the three-form, are parametrized by
\begin{equation}\label{eq:all-rank-moduli}
\operatorname{Gr}_k(\mathfrak t)/W_H.
\end{equation}
With the quotient topology induced by the commuting subspaces of $\operatorname{Gr}_k(\mathfrak p_H)$, this parametrization is a homeomorphism. The moduli space is compact and has covering dimension $k(r-k)$.
\end{theorem}

We now prove Theorem~\ref{thm:all-rank-moduli}. The polar decomposition reduces the classification of extensions of every dimension to a finite group action on a Grassmannian.

\begin{proof}[Proof of Theorem~\ref{thm:all-rank-moduli}]
Put $v=(\mu,H/\sqrt3)$ and consider the closed group
\begin{equation}\label{eq:extension-reductive-group}
L=\{l\in\mathrm{GL}(\n):l\cdot v=v,\ lR_\n=R_\n l\}.
\end{equation}
The action is change of basis on the bracket and the three-form, with differential $\pi=\Theta\oplus\rho$ as in Lemma~\ref{lem:normal-stabilizer}. For $l\in L$, write $l=qe^S$, where $q$ is orthogonal and $S=S^*$. Since $R_\n$ is symmetric and commutes with $l$, it also commutes with $l^*$, and hence with both polar factors. Set $f(t)=\|e^{t\pi(S)}v\|^2$. Equation~\eqref{eq:joint-moment} gives $f'(0)=8\tr(R_\n S)$. At $t=1$, the vector is $q^{-1}\cdot v$. Orthogonal equivariance gives its curvature operator as $q^{-1}R_\n q=R_\n$, so the same formula applies to $f'(1)$. Since $\pi(S)$ is self-adjoint, the convexity calculation of \cite[Lemma~5.1]{BoLa2020} yields
\begin{equation}\label{eq:extension-polar-convexity}
f'(0)=f'(1)=8\tr(R_\n S),\qquad
f''(t)=4\|\pi(S)e^{t\pi(S)}v\|^2\ge0.
\end{equation}
Integrating the nonnegative second derivative in \eqref{eq:extension-polar-convexity} from $0$ to $1$ gives $\pi(S)v=0$. Therefore $e^{tS}\cdot v=v$ for every $t\in\R$. Since $[S,R_\n]=0$, we have $e^{tS}\in L$ for every $t$, and hence $S\in\operatorname{Lie}(L)$. Also $q\cdot v=v$. Lemma~\ref{lem:normal-stabilizer} shows that the symmetric part of $\operatorname{Lie}(L)$ is exactly $\mathfrak p_H$. Orthogonal equivariance of $R_\n$ shows that $L\cap\mathrm O(\n)=K_H$. Consequently,
\begin{equation}\label{eq:extension-reductive-polar}
L=K_H\exp(\mathfrak p_H).
\end{equation}
Thus $L$ is a real reductive group in the sense of \cite[Section~1]{BoLa2020}. These assertions include $v=0$, when $R_\n=0$, $L=\mathrm{GL}(\n)$, and $\mathfrak p_H=\Sym(\n)$. By \cite[Proposition~A.1]{BoLa2020}, all maximal abelian subspaces of $\mathfrak p_H$ are $K_H$-conjugate. In particular, $r$ is the largest dimension of a commuting subspace of $\mathfrak p_H$.

We first prove that the group in \eqref{eq:extension-weyl-group} is finite. Suppose that $X$ belongs to the Lie algebra of $N_{K_H}(\mathfrak t)$. Then $X^*=-X$ and $[X,A]\in\mathfrak t$ for every $A\in\mathfrak t$. For $A,B\in\mathfrak t$, we have
\begin{equation}\label{eq:extension-normalizer-lie}
\langle[X,A],B\rangle
=\tr\bigl(X[A,B]\bigr)=0.
\end{equation}
Here the inner product is the Hilbert--Schmidt inner product. Taking $B=[X,A]$ gives $[X,A]=0$. Hence the normalizer and centralizer have the same Lie algebra. Their quotient is compact and has zero-dimensional Lie algebra, so it is finite.

Set
\begin{equation}\label{eq:commuting-extension-space}
\mathcal C_k=\{\af\in\operatorname{Gr}_k(\mathfrak p_H):[\af,\af]=0\}.
\end{equation}
By the classification following \eqref{eq:extension-compact-group}, the required moduli space is $\mathcal C_k/K_H$. Every $\af\in\mathcal C_k$ is contained in a maximal abelian subspace of $\mathfrak p_H$. Equation \eqref{eq:extension-reductive-polar} and \cite[Proposition~A.1]{BoLa2020} therefore show that the inclusion $\operatorname{Gr}_k(\mathfrak t)\subset\mathcal C_k$ induces a continuous surjection from \eqref{eq:all-rank-moduli} onto $\mathcal C_k/K_H$.

It remains to identify its fibers. Suppose that $\af,\mathfrak b\subset\mathfrak t$ and $q\af q^{-1}=\mathfrak b$ for some $q\in K_H$. Consider the centralizer
\begin{equation}\label{eq:extension-subspace-centralizer}
L_{\mathfrak b}=\{l\in L:lB=Bl\text{ for every }B\in\mathfrak b\}.
\end{equation}
Since the operators in $\mathfrak b$ are symmetric, the polar factors of any $l\in L_{\mathfrak b}$ also commute with $\mathfrak b$. Thus $L_{\mathfrak b}$ is closed and satisfies the same real reductivity condition. Its orthogonal subgroup is $Z_{K_H}(\mathfrak b)$ and its symmetric Lie algebra part consists of the elements of $\mathfrak p_H$ commuting with $\mathfrak b$.

Both $\mathfrak t$ and $q\mathfrak t q^{-1}$ are maximal abelian subspaces of this symmetric part. Applying \cite[Proposition~A.1]{BoLa2020} to $L_{\mathfrak b}$ gives $c\in Z_{K_H}(\mathfrak b)$ such that
\begin{equation}\label{eq:extension-centralizer-correction}
cq\mathfrak t q^{-1}c^{-1}=\mathfrak t,
\qquad
cq\af q^{-1}c^{-1}=\mathfrak b.
\end{equation}
Hence $cq\in N_{K_H}(\mathfrak t)$, and $\af$ and $\mathfrak b$ lie in the same $W_H$-orbit. This proves injectivity.

The commuting condition in \eqref{eq:commuting-extension-space} is closed, so $\mathcal C_k$ is compact. The quotient $\mathcal C_k/K_H$ is Hausdorff because $K_H$ is compact. The continuous bijection above is therefore a homeomorphism. Finally, the quotient projection from $\operatorname{Gr}_k(\mathfrak t)$ is a continuous open surjection, with each fibre containing at most $|W_H|$ points. Both spaces are compact metrizable, hence paracompact Hausdorff. Their covering dimensions agree by \cite[Theorem~4.1]{Na1960}, giving $k(r-k)$. This argument permits ineffective actions and includes $k=0$ and $k=r$.
\end{proof}

\begin{example}\label{ex:compact-stabilizer}
Let $\n=\R^7$ be abelian with its standard inner product, and use the notation $e^{ijk}=e^i\wedge e^j\wedge e^k$. Consider the standard $G_2$ three-form of \cite[Section~2.1, equation~(2.1)]{Br2006},
\begin{equation}\label{eq:compact-stabilizer-form}
H=e^{123}+e^{145}+e^{167}+e^{246}-e^{257}-e^{347}-e^{356}.
\end{equation}
Each index occurs in three summands and distinct contractions are orthogonal, so $H^2=6I$. Since $\n$ is abelian, $dH=\delta H=0$. The nilsoliton equations hold with $\lambda=-9/2$ and $D_\n=3I$, since $\Ric_\n-H^2/4=-3I/2=\lambda I+D_\n$ and $\rho(3I)H=-9H$.

For a symmetric matrix $S=(s_{ij})$, the standard norm identity in \cite[Section~2.6]{Br2006} takes the following form in our insertion convention. Let $S^\flat(X,Y)=\langle SX,Y\rangle$. If $i_H$ denotes Bryant's map on symmetric two-tensors, then $\rho(S)H=-i_H(S^\flat)/2$. Write $S=S_0+(\tr S)I/7$, with $\tr S_0=0$. The identities $i_H(\langle\cdot,\cdot\rangle)=6H$ and $\|i_H(S_0^\flat)\|_\wedge^2=8\|S_0\|^2$, together with orthogonality of the two components and $\|H\|_\wedge^2=7$, give
\begin{equation}\label{eq:compact-stabilizer-direct}
\|\rho(S)H\|_\wedge^2
=2\sum_i s_{ii}^2+4\sum_{i<j}s_{ij}^2+\left(\sum_i s_{ii}\right)^2
=2\|S\|^2+(\tr S)^2.
\end{equation}
Thus \eqref{eq:compact-stabilizer-direct} gives $\mathfrak p_H=0$. This also follows from the classical fact that the full linear stabilizer of $H$ is the compact group $G_2\subset\mathrm{SO}(7)$, as described in \cite[Sections~2.1--2.2]{Br2006}. By Corollary~\ref{cor:extension-existence}, these data admit no extension in Theorem~\ref{thm:extension} with $\dim\af>0$. Equivalently, they cannot occur as the full induced nilradical data of an algebraic generalized Ricci soliton on a nonnilpotent solvable Lie algebra. This obstruction concerns extensions with $\n$ remaining the nilradical.
\end{example}

\section*{Acknowledgements}
This work was supported by the National Natural Science Foundation of China (12571025 and 12131012) and the Natural Science Research of Jiangsu Education Institutions of China (No. 23KJB110016).

\medskip
\noindent\textbf{AI Assistance.} Generative AI tools (OpenAI ChatGPT Astra and Codex) were used under the authors' direction to assist with mathematical exploration, drafting, literature searches, and internal review of this work. The authors are responsible for independently checking the mathematical arguments and references in the final manuscript, make all final decisions concerning its contents, and take full responsibility for the paper.

\end{document}